\documentclass[a4paper,12pt]{article}

\newenvironment{proof}[1][]{\noindent {\bf Proof #1:\;}}{\hfill $\Box$}

\newtheorem{theorem}{Theorem}
\newtheorem{lemma}{Lemma}
\newtheorem{proposition}{Proposition}
\newtheorem{corollary}{Corollary}

\newtheorem{definition}{Definition}
\newtheorem{assumption}{Assumption}
\newtheorem{remark}{Remark}

\usepackage{amsmath}
\usepackage{amssymb}
\usepackage{amsfonts}
\usepackage{graphicx}
\usepackage{mathrsfs}
\usepackage{color}

\author{Mathieu Claeys$^1$, Didier Henrion$^{2,3,4}$, Martin Kru{\v{z}}{\'\i}k$^{5}$ }

\title{Semi-definite relaxations for optimal control problems with oscillation and concentration effects}

\newcommand{\R}{\mathbb{R}}
\newcommand{\N}{\mathbb{N}}
\newcommand{\PositiveMeasures}{\mathcal{M}^+}
\newcommand{\SignedMeasures}{\mathcal{M}}
\newcommand{\ContFunc}{\mathcal{C}}
\newcommand{\ubar}{u}
\newcommand{\ybar}{y}
\newcommand{\wbar}{w}
\newcommand{\diff}{d}
\newcommand{\bt}{\begin{theorem}}
\newcommand{\et}{\end{theorem}}

\newcommand{\g}{\gamma}

\renewcommand{\d}{d}

\renewcommand{\a}{\alpha}

\newcommand{\e}{{\rm e}}
\newcommand{\be}{\begin{eqnarray}}
\newcommand{\ee}{\end{eqnarray}}

\newcommand{\s}{\sigma}
\renewcommand{\e}{\varepsilon}

\renewcommand{\b}{\beta}

\newcommand{\Rn}{\R^{m}}

\newcommand{\LB}{\left\{^{^{^{}}}\right.\!\!}
\newcommand{\RB}{\!\!\left.^{^{^{}}}\right\}}
\newcommand{\rca}{\mathcal{M}}

\newcommand{\cDM}{{DM}^p(T;\R^m)}

\begin{document}

\maketitle

\footnotetext[1]{SABCA Brussels, Belgium.}
\footnotetext[2]{CNRS-LAAS, 7 avenue du colonel Roche, F-31400 Toulouse, France.}
\footnotetext[3]{Universit\'e de Toulouse, LAAS,F-31400 Toulouse; France.}
\footnotetext[4]{Faculty of Electrical Engineering, Czech Technical University in Prague,
Technick\'a 2, CZ-166 26 Prague, Czech Republic.}
\footnotetext[5]{Institute of Information Theory and Automation of the ASCR, Pod vod\'arenskou ve{\v{z}}{\'\i} 4, CZ-182 08, Prague, Czech Republic.}
\footnotetext[6]{Faculty of Civil Engineering, Czech Technical University in Prague, Th\'akurova 7, CZ-166 29 Prague,
Czech Republic.} 

\begin{abstract}
Converging hierarchies of finite-dimensional semi-definite relaxations have
been proposed for state-constrained optimal control
problems featuring oscillation phenomena, by relaxing controls as Young measures. These
semi-definite relaxations were later on extended to
optimal control problems depending linearly on the
control input and typically featuring concentration
phenomena, interpreting the control as a measure of
time with a discrete singular component modeling
discontinuities or jumps of the state trajectories. In this
contribution, we use measures introduced originally by
DiPerna and Majda in the partial differential equations
literature to model simultaneously, and in a unified
framework, possible oscillation and concentration
effects of the optimal control policy. We show that
hierarchies of semi-definite relaxations can also be
constructed to deal numerically with nonconvex optimal control
problems with polynomial vector field and semialgebraic
state constraints.
\end{abstract}

\section{Introduction}
This paper is devoted to a numerical method for solving optimal control problems which may exhibit concentrations and/or oscillations.  
Such problems naturally appear in many
applications, for instance in controlling space shuttles, 
in impulsive control theory, see e.g. \cite{blaquire},
\cite{bryson-ho}, or  \cite{getz-martin}.

 A typical example  how concentration effects enter a minimization problem is the following: 
\begin{equation}\label{preface1}
        \begin{array}{ll}
\mbox{inf} & \displaystyle{\int_0 ^1 (t-1/2)^2|u(t)|\,\d t } \\
\mbox{s.t.} & \dot y = u,
                      \ \ y(0) = 0,\  y(1)=1, \\
        & y\in W^{1,1}(0,1),\ \ u\in L^1(0,1),
        \end{array}
\end{equation}
where $u$ is a control living in the Lebesgue space of square integrable
functions $L^2(0,1)$, $y$ is the corresponding state living in the
Sobolev space $W^{1,1}(0,1)$ of functions whose (distributional)
derivative belongs to the Lebesgue space of integrable functions $L^1(0,1)$,
and the dot denotes time derivative.

Roughly speaking, an optimal control must here balance between transfering the state from $0$ to $1$ and remaining constant ``as much as possible''.   Therefore, the optimal state should switch 
between two constants. As the function  $t\mapsto (t-1/2)^2$ attains its minimum   on $[0,1]$ at $t=1/2$, it is desirable to switch from zero to one in a vicinity of $t=1/2$. Having this in mind, it is easy to check that in problem (\ref{preface1}) the infimum is zero but it cannot be attained. Changing suitably the weight function $t\mapsto(t-1/2)^2$, the jump point can appear anywhere in $[0,1]$, and in particular, also on the boundary.

 Concentrations do not have to occur only at isolated
points,  but they can be smeared out  along the whole interval.
This can be demonstrated on the following  problem:
\begin{equation}\label{zadani3}
        \begin{array}{ll}
\mbox{inf} & \displaystyle{\int_0 ^1 \left(\frac{u(t)^2}{1+u(t)^4}+(y(t)-t)^2\right)\,\d t } \\[3mm]
\mbox{s.t.} & \displaystyle{\dot y = u},\ \ y(0) = 0, \\
        & y\in W^{1,1}(0,1),\ \ u\in L^1(0,1), \ \ u \geq 0.
        \end{array}
\end{equation}
Here, the cost favors instantaneous controls of null or very large magnitude, while averaging out to $1$. The infimum of (\ref{zadani3}) is
$0$. Indeed, consider the control 
$$
 u^k(t)=\left\{\begin{array}{ll}
                         k & \mbox{ if $t\in [\frac{l}{k}-\frac{1}{2k^2},
                             \frac{l}{k}+\frac{1}{2k^2}]$}\\ 
                          0 & \mbox{ otherwise}
                          \end{array}
                   \right. $$
where $k\in\N$ is large enough and $1\le l\le k-1$. The corresponding state then reads  $y^k(t)=\int_0^t u^k(s)\,\d s$. 
An easy calculation shows that with this sequence $(y^k,u^k)_{k\in\N}$ the cost in problem (\ref{zadani3})
can approach zero as closely as desired,
so that the infimum is zero, while there is obviously no minimizer.      

In this paper, we propose a numerical method for solving those types of problems in a unified framework.
The high-level methodology we follow is now familiar \cite{sicon,impulse}:
\begin{itemize}
\item First, a suitable relaxed concept of control is assumed, so as to guarantee the existence
of minimizers under sufficiently broad assumptions.
\item Second, as the problem still depends non-convexly on trajectories, whether via dynamics,
constraints or the cost, suitable variants of occupation measures are introduced,
so as to lift the optimization problem as a linear program over a measure
space.
\item Finally, if problem data are polynomial, directly or via algebraic lifts, this linear
program can be solved via the paraphernalia of moment-sum-of-squares semidefinite programming
relaxations (moment-SOS relaxations for short).
\end{itemize}

For problems with bounded controls, which may give rise to fast oscillations, appropriate
relaxed concepts of controls were developed decades ago, by the likes of Young \cite{Young},
Fillipov, Warga, Gamkrelidze \cite{Gamkrelidze} and many others, see e.g. \cite[Part III]{Fattorini}.
The equivalence
between the relaxed control problem and the measure LP was then proven in
\cite{Vinter1978Equivalence,Vinter1993ConvexDuality}, see also
\cite{GaitsgoryQuincampoix2009}. More recently, \cite{sicon}
proposed to solve this linear program via moment-SOS relaxations.
For unbounded controls entering affinely in the problem, which may give rise to concentration
effects, various weak concepts of solution have been proposed, where the control
is now a measure of time \cite{Neustadt,Gamkrelidze} in the simplest case, with more intricate cases requiring
appropriate graph completions \cite{Bressan1988graphcompl} to ensure well-posedness of the relaxed program.
Variants of occupation measures were then introduced \cite{impulse,claeys2014thesis} so as to solve the problem via
moment-SOS relaxations, optimizing also w.r.t. all possible graph completions.

For unbounded controls entering non-linearly in the control problem, minimizing sequences
may give rise to both oscillation and concentration phenomena. DiPerna and
Majda \cite{diperna1987oscillations} have introduced measures allowing to capture the limit of admissible control
sequences. Those results were then used in \cite{Kruzik1998optimization} as the basis of a numerical method to
solve unconstrained, convex (in the trajectories) problems. In this paper, we introduce
appropriate occupation measures to relax this class of problem as a linear program over
measures, and develop the subsequent moment-SOS relaxations resulting from this formulation.
This allows to attack non-convex problems, and as importantly for practical applications,
handle state constraints seamlessly.

\subsection*{Contribution}
The method of moments has already been used in relaxations of optimal control problems
with oscillations in \cite{cdc05} and then \cite{sicon} using occupation measures and moment-SOS techniques.
See also \cite{Meziat-etal} and its calculus of variations counterpart \cite{Meziat-etal-2}
which do not use occupation measures but time-dependent moments and hence time-discretization,
and which do not account for state constraints.
Occupation measures and moment-SOS techniques were then used in \cite{impulse,claeys2014thesis}
to cope with concentrations. In this paper, these techniques are extended 
(to our best knowledge for the first time) for optimal control problems with oscillations and concentrations.  
As a by-product, we describe explicitly how compactification techniques can be used
to deal numerically with non-compact sets in the moment-SOS approach to optimal control.

\subsection*{Organization of the paper}

The paper is organized as follows. Section \ref{sec:prelim} introduces the problem and recalls key results about DiPerna-Majda
measures. Section \ref{sec:linearPrograms} develops the primal linear program on measures, and also present its
Hamilton-Jacobi Bellman dual. Section \ref{sec:momentRelax} details how this primal-dual pair can be solved via moment-SOS
relaxations and semidefinite programming. Section \ref{sec:extensions} presents several extensions to the approach, relegated to this section
for clearer exposition. Finally, \ref{sec:examples} presents several relevant examples.

\subsection*{Notations}

Let $\SignedMeasures(X)$ denote the vector space of finite, signed, Radon measures supported on  an Euclidean subset
$X \subset \R^n$,  equipped with the weak-star topology,  see e.g.  \cite{Royden} for background material.
Let $\PositiveMeasures(X)$ denote the cone of non-negative measures in $\SignedMeasures(X)$.
For a continuous function $f \in  \ContFunc(X)$, we denote by $\int_{X} \! f(x) \, \mu( \d x)$ the integral of $f$ w.r.t. the measure $\mu \in \SignedMeasures(X)$.
When no confusion may arise, we use the duality bracket notation $\langle f, \mu \rangle = \int f \mu$ for the integral to simplify exposition and to insist on the duality relationship between $ \ContFunc(X)$ and $\SignedMeasures(X)$ on compact $X$.
 The Dirac measure supported at $x^*$, denoted by $\delta_{x^*}$, is the measure for which $\langle f, \delta_{x^*} \rangle = f(x^*)$ for all $f \in \ContFunc(X)$.  The indicator function of set $A$, denoted by $I_A(x)$,  is equal to one if $x \in A$, and zero otherwise. The space of probability measures on $X$,  the subset of $\PositiveMeasures(X)$ with mass $\langle 1,\mu \rangle = 1$, is denoted by $\mathcal{P}(X)$. 

Let $T:=[t_0,\:t_f]$ denote a time interval, with $t_f > t_0 \geq 0$.
We use standard notation for Lebesgue and Sobolev spaces, i.e. $L^p(T;\R^m)$ and $W^{1,1}(T;\R^m)$. Due to nonreflexivity of $W^{1,1}(T;\R^m)$ we are forced to enlarge 
the space of states (trajectories) to $BV(T;\R^m)$, the space  of functions of bounded variations on $T$. However, the standard definitions of the BV space does not  take into account of jumps of function values at the boundary of the interval \cite{Ambrosio2000BV} as it is typically defined on an open set.  Nevertheless, we can choose a small $\e>0$ and extend any absolutely continuous function $y\in W^{1,1}(T;\R^m)$ to $(t_0-\e;t_f+\e)$ such that this extension $\tilde y$ satisfies  $\tilde y(t)=y(0)$ for $t\in [t_0-\e;t_0)$ and $\tilde y(t)=y(t_f)$ for $t>t_f$.  Then we define 
$BV(T;\R^m)$ as the set of restrictions to $T$ of the weak-star (in the sense of BV) closure of the set $\{\tilde y:\, y\in W^{1,1}(T;\R^m)\}$. In fact, this definition coincides with the so-called Sou\v{c}ek space $W^{1,\mu}(T;\R^m)$ introduced in \cite{soucek} and used e.g. in \cite{Roubicek}. 

For multi-index $\alpha \in \N^n$ and vector $x\in \R^n$, we use the notation $x^\alpha := \prod_{i=1}^n x_i^{\alpha_i}$ for a monomial.
We denote by $\N^n_m$ the set of vectors $\alpha \in \N^n$ such that $\sum_{i=1}^n \alpha_i \leq m $. 
The moment of multi-index $\alpha \in \N^n$ of measure $\mu \in \PositiveMeasures(X)$ is then defined as the real $y_\alpha = \langle x^\alpha, \mu \rangle$.  
A multi-indexed sequence of reals $(z_\alpha)_{\alpha \in \N^n}$ is said to have a representing measure on $X$ if there exists $\mu \in \PositiveMeasures(X)$ such that $z_\alpha = \langle x^\alpha , \mu \rangle$ for all $\alpha \in \N^n$.
Let $\R[x]$  denote the ring of polynomials in the variables $x \in \R^n$, and let $\deg p$ denote the (total) degree of polynomial $p$.
A subset of $\R^n$ is basic semi-algebraic if it is defined as the intersection of finitely many polynomial inequalities, 
namely $\lbrace x \in \R^n : \; g_i(x) \geq 0, \, i = 1 \ldots n_{X}\rbrace$ with $g_i(x) \in \R[x], \, i = 1 \ldots n_{X}
$.  

\section{Preliminaries}
\label{sec:prelim}

\subsection{DiPerna-Majda measures}

As the motivating examples reveal, the infimum of the optimal control problems might not be attained.
For that reason, we have to construct a suitable locally compact convex hull
of the Lebesgue spaces \cite{Roubicek}. As already  mentioned, we will use a special
extension proposed by DiPerna and Majda \cite{diperna1987oscillations,Kruzik-Roubicek,Roubicek-Kruzik-2000}.

Let $S^{m-1}$ denote the unit sphere in $\R^m$ and let
\begin{eqnarray}\label{compsferou}
{\cal R}&=& \LB v\!\in\!\ContFunc_b(\R^m) \: :\: \exists \: v_0\!\in\!\ContFunc_b(\R^m),\ v_1\!\in\!
\ContFunc(S^{m-1}),\ c\in\R:\nonumber\\ 
& & \ \ \ \lim_{|u|\to\infty} v_0(u)=0, \:\: v(u) = c+
v_0(u)+v_1\left(\frac{u}{|u|}\right) 
\frac{|u|^p}{1+|u|^p}\RB
\end{eqnarray}
denote the complete separable  subring of the ring $\ContFunc_b(\R^m)$ of bounded continuous functions on $\R^m$.
The corresponding compactification of $\R^m$, denoted $\g\R^m$, is then homeomorphic
with the unit ball or equivalently the simplex in $\R^m$. This means that every $v\in{\cal R}$ admits a uniquely
defined continuous extension on $\g\R^m$ (denoted then again by
$v$ without any misunderstanding) and conversely, for every $v\in
\ContFunc(\g\R^m)$, the restriction on $\R^m$ lives in ${\cal R}$.
 
Let $\s \in \PositiveMeasures(T)$, and let $L^{\infty}_{\rm w}(T,\s;\rca(\g \R^m))$ denote the Banach space of all weakly $\s$-measurable\footnote{This means
that for any $w\in {\cal R}$, the mapping $t \in T\mapsto\int_{\g\R^m}
w(\ubar) \nu(\diff \ubar | t) \in \R$ is $\s$-measurable in the usual sense.}
$\s$-essentially bounded mappings from $T$ to the set of Radon
measures $\rca(\g\R^m)$ on $\g\R^m$. 
Let ${\cal U}(T,\sigma;\gamma\R^m)$ denote
the subset of $L^{\infty}_{\rm w}(T,\s;\rca(\g \R^m))$ consisting
of the mappings $\nu:t\mapsto\nu(\diff u | t)$ such that
$\nu(\diff u| t)$ is a probability measure on $\g\R^m$ for $\s$-almost all
$t\in T$.

DiPerna and Majda \cite{diperna1987oscillations}  showed that, given a bounded
sequence $(u_k)_{k\in\N}$ in $L^p(T;\R^m)$, $1\le p<+\infty$, there exists a
subsequence (denoted by the same indices) and a measure  $\eta\in\rca^+(T\times\gamma\R^m)$ such that  for any 
$h_0\!\in\! \ContFunc(T\times\gamma\R^m)$,
\be \label{defDM10}
\lim_{k\to\infty}\int_T h(t,u_k(t))\d t\ =
\int_T\!\int_{\g\R^m}h_0(t,u) \d\eta(t,u)\ ,
\ee 
where $h(t,u)=h_0(t,u)(1+|u|^p)$.  DiPerna-Majda measures are precisely all the  measures $\rca^+(T\times\gamma\R^m)$
that are attainable in the sense of \eqref{defDM10} by some sequence in $L^p(T;\R^m)$.

For our purposes, it will be convenient to disintegrate a DiPerna-Majda measure $\eta\in\rca^+(T\times\gamma\R^m)$
as a product of a time marginal $\s\in\rca^+(T)$ and a state conditional 
$\nu\in{\cal U}(T,\sigma;\gamma\R^m)$, i.e.
\[
\eta(dt,du) =\nu(du|t)\sigma(dt).
\]
By this, we mean that for any $h_0\!\in\! \ContFunc(T\times\gamma\R^m)$,
\be \label{defDM2}
\lim_{k\to\infty}\int_T h(t,u_k(t))\d t\ =
\int_T\!\int_{\g\R^m}h_0(t,u) \nu(\d u |t) \s(\d t),
\ee 
where $h$ is related to $h_0$ as before.  
We say that such a pair
$(\sigma, \nu)\in \rca(T)\times {\cal Y}(T,\s;\g\R^m)$ 
is attainable by a sequence 
$(u_k)_{k\in\N}\subset L^p(T;\R^m)$. The set of all attainable
pairs $(\sigma,\nu)$ is denoted by $\cDM$.

Given $(\sigma,\nu) \in \cDM$ we denote by $\|(\sigma,\nu)\|$ the norm
of total variation, i.e. the mass of the measure $\eta(dt,du):=\nu(du|t)\sigma(dt) \in 
\rca^+(T\times\gamma\R^m)$.

\subsection{Control relaxations}

We consider the following class of optimal control problems:
\begin{eqnarray}\label{eq:OCP}
        \begin{array}{rll}
     &\mbox{minimize} & \displaystyle{\int_T l(t,y,u) \d t
     } \\ 
&\mbox{s.t.} 
        & \displaystyle{\dot y = f(t,y,u),} \qquad  y(0)=y_0, \\ 
        & & y\!\in\!W^{1,1}(T;Y),\ \ \ \ u\!\in\!L^p(T;\R^m),
        \end{array}
\end{eqnarray}
where $Y$ is a given subset of $\Rn$.
For this section and the next, we make the following integrability and growth conditions on the Lagrangian and dynamics:

\begin{assumption}
\label{th:assGrowth}
Lagrangian $l(t,y,u):= a(t,y,u)+b(t,u)$ and dynamics $f(t,y,u):= c(t,y,u)+d(t,u)$ are such that
 $a:T\times Y \times\R^m\to\R$ is a Carath\'eodory
function\footnote{This means that $a(t,\cdot,\cdot):Y\times\R^m\to\R$ is
continuous for almost every $t\in T$ and $a(\cdot,\ybar,\ubar):T\to\R$ is
measurable for all $\ybar$ and $\ubar$.},
$c:T\times Y \times\R^m\to\R^n$ is a Carath\'eodory
function,  and $b:T\times\R^m\to\R$ is continuous, while  satisfying 
\be\label{ass1}
& & \max(|a(t,\ybar,\ubar)|,|c(t,\ybar,\ubar)|)\le
\a_{1+\e}(t)+\b(|\ybar|^{1/\e}+|\ubar|^{p/(1+\e)}), \\
\label{ass2} 
& & b_0\in \ContFunc(T;{\cal R}),\ \ \ \ \ \mbox{ where }\ 
b_0(t,\ubar):=b(t,\ubar)/(1+|\ubar|^p),\\
\label{ass3}
& & |c(t,\ybar,\ubar)|\le (\a_1(t)+\b| \ubar |^{p})(1+| \ybar |), \\
\label{ass4}
& & |a(t,\ybar_1,\ubar)-a(t,\ybar_2,\ubar)|\le
(\a_1(t)+\b|\ybar_1|^{1/\e}+\b|\ybar_2|^{1/\e}
+\b|\ubar|^{p/(1+\e)})|\ybar_1-\ybar_2|,\\
\label{ass5} 
& & |c(t,\ybar_1,\ubar)-c(t,\ybar_2,\ubar)|\le
(\a_1(t)+\b|\ybar_1|^{1/\e}+\b|\ybar_2|^{1/\e}
+\b|\ubar|^{p/(1+\e)})|\ybar_1-\ybar_2|,\\
\label{ass6} 
& & d_0\in \ContFunc(T;{\cal R}^n),\ \ \ \ \ \mbox{ where }\ 
d_0(t,\ubar):=d(t,\ubar)/(1+|\ubar|^p),\\
\label{ass7}
& & a(t,\ybar,\ubar)+b(t,\ubar)\ge\delta |\ubar|^p\ ,
\ee
with some $\e>0$, $\delta\ge 0$, $\b\in\R$, $\a_q\in L^q(T)$. 
\end{assumption}

Note that if $\delta=0$ in \eqref{ass7}, we must slightly adapt the discussion in the paper by bounding the set of admissible controls in $L^p(T;\R^m)$. This will result in an additional linear constraint in the measure LP and the moment relaxations, as explained in Section \ref{sec:extensions}. Moreover, we really can admit only measurability in the first variables of $a$ and $c$. The reason is their growth in the control variable which is smaller than $p$, namely at most $p/(1+\e)$.

As shown  in \cite{Kruzik1998optimization}, DiPerna-Majda measures allow for the
following relaxation of \eqref{eq:OCP}, which we refer to the ``strong'' problem (by opposition to a
weak problem to defined later on):
\begin{equation}
\label{eq:strongProblem}
        \begin{array}{rll}
&\mbox{minimize} &\displaystyle{\int_T \int_{\g\R^m} \frac{l(t,y(t),\ubar)}{1+|\ubar|^p} \:\nu(\d \ubar |t)\s(\d t)} \\[3mm]
&\mbox{s.t.} &
y(\tau) = y(t_0) + \displaystyle \int_{t_0}^\tau \int_{\g\R^m}\frac{f(t,y(t),\ubar)}{1+|\ubar|^p}\nu(\d \ubar |t)\s(\d t), \:\: \tau \in T,\\[2mm] 
     & & y\in BV(T;\R^m),\ \ \ \ \
         (\s, \nu)\in\cDM,\ \\
        \end{array}
\end{equation}
where the differential equation is now expressed in integral form.

The  natural embedding
\begin{equation}
\label{eq:naturalImbedding}
\begin{array}{rclcl}
i & : & L^p(T;\R^m) & \to & \cDM \\
&& u & \mapsto & (\sigma,\:\nu) :=((1 + \lvert u(t) \rvert^p)\d t,\:\delta_{u(t)}(\diff \ubar | t))
\end{array}
\end{equation}
establishes readily  that \eqref{eq:strongProblem} is really an extension of \eqref{eq:OCP}. In fact, the following two assertions justify that \eqref{eq:strongProblem} is
actually a legitimate relaxation of \eqref{eq:OCP}, see \cite{Kruzik1998optimization} for proofs.

\begin{proposition}\label{proposition31} 
Let \eqref{ass1}--\eqref{ass3}, \eqref{ass5}, and \eqref{ass6} be
valid. Assume $Y = \Rn$ and a fixed initial condition $y_0 \in \Rn$.
For $k\in\N$, let $y^k$ solve
\be\label{approxstate}
\dot y^k=f(t,y^k,u^k),\  \ y^k(t_0)=y_0,
\ee
and let the sequence $(u^k)_{k\in\N}$ attain $(\s,\nu)$ in the sense of \eqref{defDM2}.
Then $y^k \to y$ weakly star in ${\rm BV}(T;\Rn)$ and
$y^k(t_f)\to y(t_f)$, with $y$ uniquely defined by
\be\label{relaxstate}
y(\tau) = y_0 + \int_{t_0}^{\tau} \int_{\g\R^m}\!\frac{f(t,y(t),\ubar)}{1+|\ubar|^p}\nu(\d \ubar | t)\s(\d t),\:\: \tau \in T.
\ee
\end{proposition}

Hence, we can define a map $\pi:\cDM\to BV(T;\R^m)$ such that 
$\pi(\s,\nu)=y$ where $y$ is a solution to \eqref{relaxstate}.

\begin{proposition}\label{proposition}
Let \eqref{ass1}--\eqref{ass7} be valid. Assume $Y = \Rn$, a fixed initial condition $y_0$ and a free final state. Then
the infimum in the original problem \eqref{eq:OCP} is equal to the infimum in the
strong problem \eqref{eq:strongProblem}, which is attained. Moreover, 
every solution $(\sigma,\nu) \in\cDM$ to the strong problem \eqref{eq:strongProblem} is attainable in the sense of
\eqref{defDM2} by a minimizing sequence for the original problem \eqref{eq:OCP}, and,
conversely, every minimizing sequence for the original problem \eqref{eq:OCP} contains
a subsequence converging in the sense of \eqref{defDM2} to a solution of the
strong problem \eqref{eq:strongProblem}. 
\end{proposition}

\begin{remark}
If we consider an additional  constraint on state trajectories, i.e. $y(t)\in Y$ for some compact $Y\subset\R^n$,
Proposition~\ref{proposition} does not hold anymore in general. Indeed, 
consider \eqref{preface1}, our first example in the introduction, with $Y=\{0,1\}$. Then the infimum of  (\ref{eq:OCP})
is infinity because the admissible set of states is empty, while the infimum of (\ref{eq:strongProblem}) is zero.
Different boundary conditions and/or an explicit bound on the norm of the control might yield a similar relaxation gap.
On the other hand, we know that if $y\in BV(T;\R^m)$ is an optimal state (a solution) then $y\in L^\infty(T;\R^m)$. Therefore, there is a closed  ball  $B(0,R)$ in $\R^m$ centered at the origin and of radius $R$ which contains $y$.
Therefore, taking $Y:=B(0,R)$ for $R>0$ large enough, it is  always possible to  obtain  a solution on the whole $\R^m$.  
\end{remark}

Proposition \ref{proposition31} establishes uniqueness of the Cauchy problem for the relaxed differential equation. This motivates the following handy definition for the remainder of the paper.
\begin{definition}[Relaxed arc]
A triplet $(\sigma,\:\nu,\:y)$ is called a relaxed arc for strong problem \eqref{eq:strongProblem}
if $(\s, \nu)\in\cDM$ and  $y\in BV(T,Y)$ satisfy \eqref{relaxstate}, as well as the boundary conditions of \eqref{eq:strongProblem}.
\end{definition}

We can readily prove the following result:
\begin{proposition}
Assume that there is a relaxed arc  $(\sigma,\:\nu,\:y)$ for problem \eqref{eq:strongProblem}
such that $\|(\sigma,\:\nu)\| \leq c$ for some $c>0$.
If $y \in Y$  for every solution $\dot y=c(t,y,u)+d(t,u)$, $y(t_0)=y_0$  where $u\in L^p(T;\R^m)$
 is such that $\|i(u)\| \le c$ then the corresponding relaxation result holds as in Proposition~\ref{proposition}.
\end{proposition}

\section{Weak linear program}
\label{sec:linearPrograms}

This section constructs a weak linear program for relaxing optimal control problem \eqref{eq:OCP}, and explicits its dual. For ease of exposition, we assume that both boundary conditions in \eqref{eq:OCP} are prescribed (see Section \ref{sec:extensions} on how to relax this assumption). We also assume that $Y$ is compact, since this is required anyway in the next section for manipulating measures by their moments. Finally, we assume $\delta > 0$ in Assumption \ref{th:assGrowth}, so that we do not need to explicitly handle a bound on the $L^p$ norm of the control (see Section \ref{sec:extensions} again for how such a constraint would modify the discussion below).

\subsection{Measure LP}
\label{sec:measureLP}
We now define occupation measures based on DiPerna-Majda measures, in much the same
way occupation measures can be constructed from Young measures, see e.g. \cite{Vinter1978Equivalence,sicon}.
Problem \eqref{eq:strongProblem} has relaxed control $(\sigma, \nu)$ as decision variables, both measures, but also 
the state trajectory $y$, considered now as a function of bounded variation. As such, the problem is not fully expressed as
a linear program on measures, and cannot yet be solved by our moment-SOS approach. In
this subsection, we therefore embed \eqref{eq:strongProblem} into such a linear program, via an appropriate
space-time reparametrization of the solution during the possible jumps.
First of all, notice that the ODE of strong problem \eqref{eq:strongProblem} is a measure-driven differential
equation. Let $\sigma=\sigma_C+\sigma_D$ denote the Lebesgue decomposition of measure $\sigma$,
with its continuous (absolutely or singularly w.r.t. the Lebesgue measure) part $\sigma_C$ and its discrete part $\sigma_D$
supported on at most countably many points of $J:=(t_j)_{j\in\N} \subset T$.
When $\sigma$ is continuous, so is the state trajectory $y$. When it is discrete, one may construct a
properly defined space-time reparametrization of  the state trajectory ``during'' the jumps,
see \cite{motta1995space,miller2003impulsive,dykhta2010hamilton}:

\begin{definition}[Space-time reparametrization]
\label{th:defReparam}
A function $y \in BV (T;\R^n)$  is a
solution to the ODE in \eqref{eq:strongProblem} if it satisfies
\[
y(\tau^+) = y(t_0^-) + \int_{t_0}^{\tau} \int_{\gamma \R^m}  \frac{ f(t,y(t),\ubar)}{1 + \lvert \ubar \rvert^p} \, \nu( \diff \ubar |t) \, \sigma_C(\diff t) 
+ \sum_{s\in J  \cap [t_0,\tau]} \left( y(s^+) -  y(s^-) \right)
\]
where $y(s^+)$ and $y(s^-)$ are boundary conditions of the following ODE:
\begin{equation}
\label{eq:limitingSystem}
\begin{gathered}
\dot y_s(t) =  \int_{\gamma \R^m} \frac{ f(s,y_s(t),\ubar)}{1 + \lvert \ubar \rvert^p} \, \nu (\diff \ubar | t), \quad t \in [0, \sigma_D(\lbrace s \rbrace )] \\
y_s(0) = y(s^-), \quad y_s(\sigma_D(\{s\})) = y(s^+)
\end{gathered}
\end{equation}
with $y_s$ a fictitious state that evolves ``during'' the jump at $s\in J$.

\end{definition}
Here $y(t^\pm)$ represent left and right limits of $y(t)$, and we implicitly extend $y(t)$ by a constant
function outside of the interval $T$ should jump discontinuities arise at its boundaries.
Note that during jumps, we follow for the rest of the paper the convention that the state $y_s$ remains
in $Y$, that is, we forbid arbitrarily fast violations of the state constraints
during jumps.
Given Proposition~\ref{proposition31}, the mere definition of relaxed control $(\sigma,\nu)$ guarantees
the uniqueness of $y$, and hence of the space-time reparametrization as well. The reason
to introduce Definition~\ref{th:defReparam} is therefore not motivated by a desire for a well-defined concept of
trajectory as in impulsive optimal control, but as a necessary step for defining an appropriate concept
of occupation measure. Indeed, fix an admissible relaxed arc $(\sigma,\nu,y)$ and its associated
space-time reparametrization $(y_s)_{s \in J}$, and define a measure
\begin{equation}
\xi(B|\ubar, s) = \begin{cases} \delta_{y(s)}(B) & \text{if}\: s \notin J \\
\int_0^{\sigma(\lbrace s \rbrace)} \frac{I_{B}(y_s(t))}{\sigma(\lbrace s \rbrace )} \, \d t &  \text{if}\: s \in J \end{cases}
\end{equation}
for any Borel set $B \subset Y$, where $I_B(y)$ denotes the indicator function equal to $1$ if $y \in B$
and $0$ otherwise. Note that the normalization by $\sigma(\lbrace s  \rbrace )$ ensures
that $\xi(B|s)$ is a probability measure. The notation $\xi(dy | \ubar,s)$ indicates that
$\xi$ is a conditional probability measure depending on control $u$ (through the DiPerna-Majda
measure $(\sigma,\nu)$) and time $s$.

\begin{definition}[Occupation measure]
\label{th:defOccpMeas}
The occupation measure $\mu \in \PositiveMeasures(T \times Y \times \gamma \R^m)$ associated to
a given admissible relaxed arc $(\sigma, \nu, y)$  is defined by
\begin{equation}
    \diff \mu(t, \ybar, \ubar) := \xi(\diff \ybar | \ubar, t) \, \nu(\diff \ubar | t) \, \sigma(\diff t)
\end{equation}
\end{definition}

The following essential property reveals that the ODE in \eqref{eq:strongProblem} gives rise to linear constraints:
\begin{proposition}
\label{th:weakConstraints}
Let $\mu$ be the occupation measure associated to admissible relaxed arc $(\sigma, \nu, y)$.
Then for all test functions $v(t,\ybar) \in \ContFunc^1(T \times Y)$ it holds:
\begin{equation}
\label{eq:weakDynamics}
\langle \frac{\partial v}{\partial t} \frac{1}{1 + \lvert \ubar \rvert^p}+ \frac{\partial v}{\partial \ybar} \frac{ f(t,\ybar,\ubar)}{1 + \lvert \ubar \rvert^p}, \mu \rangle = v(t_f,y(t_f^+)) - v(t_0,y(t_0^-)).
\end{equation}
\end{proposition}

\begin{proof}
By the chain rule in BV (\cite[Th.~3.96]{Ambrosio2000BV}), differentiating such test functions along the admissible trajectory leads to
\begin{equation}
\begin{gathered}
v(t_f,y(t_f^+)) - v(t_0,y(t_0^-)) = \int_T \diff v(t,y(t)) = \\
\underbrace{\int_T \! \frac{\partial v}{\partial t}(t,y(t)) \, \diff t}_{:=a_1}
+ \underbrace{\int_T \! \frac{\partial v}{\partial \ybar}(t,y_C(t)) \, \diff y_C(t)}_{:=a_2}
+ \underbrace{\sum_{s \in J} \left( v(s,y(s^+)) - v(s,y(s^-)) \right)}_{:=a_3}
\end{gathered}
\end{equation}
where $y_C$ denotes the continuous part of the Lebesgue decomposition of $y$ w.r.t. time.
By \eqref{defDM2} with $v=g=1$, $\diff t = \int_{\gamma \R^m} \frac{1}{1 + \lvert \ubar \rvert^p} \nu(\diff \ubar | t) \, \sigma(\diff t)$, so that the first term is equal to
\begin{equation}
a_1 = \langle  \frac{\partial v}{\partial t}(t,\ybar)  \frac{1}{1 + \lvert \ubar \rvert^p}, \mu \rangle.
\end{equation}
By Definition~\ref{th:defOccpMeas} and the fact that $\diff y(t) = \int_{\gamma \R^m} \frac{1}{1 + \lvert \ubar \rvert^p} \nu(\diff \ubar | t) \, \sigma(\diff t)$,
the second term is  equal to
\begin{equation}
a_2 = \langle \frac{\partial v}{\partial \ybar} \frac{ f(t,\ybar,\ubar)}{1 + \lvert \ubar \rvert^p}, \mu_C \rangle,
\end{equation}
where $\mu_C$ is the continuous part of the Lebesgue decomposition of $\mu$ w.r.t. time.
Finally, the last term can now be evaluated along graph completions using \eqref{eq:limitingSystem}:
\begin{equation}
\begin{aligned}
a_3 & = \sum_{s \in J} \int_0^{\sigma(\lbrace s \rbrace )} \diff v(t) \\
       & =  \sum_{s \in J} \int_0^{\sigma(\lbrace s \rbrace )}  \int_{\gamma \R^m} \frac{\partial v}{\partial \ybar}(s,y_s(t)) \, \frac{ f(s,y_s(t),\ubar)}{1 + \lvert \ubar \rvert^p} \, \nu( \diff \ubar | s) \, \d t  \\
&=  \sum_{s \in J} \int_{Y}  \int_{\gamma \R^m} \frac{\partial v}{\partial \ybar}(s,\ybar) \, \frac{ f(s,\ybar,\ubar)}{1 + \lvert \ubar \rvert^p} \,  \xi(\diff \ybar | \ubar, s) \, \nu(\diff \ubar | s)  \, \sigma(\lbrace s \rbrace ) \\  
    & =\langle   \frac{\partial v}{\partial \ybar} \frac{ f(t,\ybar,\ubar)}{1 + \lvert \ubar \rvert^p}, \mu_D \rangle,
\end{aligned}
\end{equation} 
where $\mu_D$ is the discrete part of the Lebesgue decomposition of $\mu$ w.r.t. time.
Since $\mu=\mu_C+\mu_D$, this concludes the proof.
\end{proof}

Proposition \ref{th:weakConstraints} suggests to relax strong problem \eqref{eq:strongProblem} as a linear program on measures, called hereafter the ``weak'' problem:
\begin{equation}
\label{eq:weakProblem}
\begin{aligned}
p_W^* = \inf_\mu \;\; & \langle \frac{l(t,\ybar,\ubar)}{1+ \lvert \ubar \rvert^p}, \mu \rangle  \\
\text{s.t.} \;\; &  \langle \frac{\partial v}{\partial t} \frac{1}{1+\lvert \ubar \rvert^p}  + \frac{\partial v}{\partial \ybar} \frac{f(t,\ybar,\ubar)}{1+\lvert \ubar \rvert^p}, \mu \rangle=v(t_f,y(t_f^+)) - v(t_0,y(t_0^-)), \quad \forall v \in \ContFunc^1(T \times Y)\\
& \mu \in \PositiveMeasures(T \times Y \times \gamma \R^m).
\end{aligned}
\end{equation}
Obviously, $p_W^*$ is smaller than or equal to the infimum of \eqref{eq:strongProblem}.
Note that it is conjectured that the values actually agree, as is the case for bounded controls,
see \cite{Vinter1978Equivalence,Vinter1993ConvexDuality}. In the next section, \eqref{eq:weakProblem} is further relaxed to obtain a finite-dimensional,
tractable problem. The absence of a relaxation gap is then simply tested a posteriori by
observing that the finite-dimensional solution converges to a solution of \eqref{eq:OCP}. See also the examples in Section~\ref{sec:examples}.

\subsection{Dual conic LP}
\label{sec:duality}

Before investigating the practical implications on semi-definite relaxations of measure LP \eqref{eq:weakProblem},
we explore its conic dual. This is an interesting result in its own right for so-called
``verification theorems'' which supply necessary and sufficient conditions in the form of
more traditional Hamilton-Jacobi-Bellman (HJB) inequalities.
In addition, practical numerical resolution of the semidefinite programming (SDP)
relaxations by primal/dual interior-point methods \cite{SeDuMi} implies that a strengthening of this
dual will be solved as well, as shown in Section~\ref{sec:SOS}.

In this section, it is first shown that the solution of \eqref{eq:weakProblem} is attained whenever an admissible
solution exists. Then, the dual problem of \eqref{eq:weakProblem}, in the sense of conic duality, is presented.
This leads directly to a HJB-type inequality. Although, the value of this problem might not
be attained, it is however shown that there is no duality gap between the conic programs.

First of all, we establish that whenever the optimal value of \eqref{eq:weakProblem} is finite, there exists a
vector of measures attaining the value of the problem:
\begin{lemma}
\label{th:solAttained}
If $p_W^*$ is finite in problem \eqref{eq:weakProblem}, there exists an admissible $\mu$ attaining the
infimum, viz. such that $\langle \frac{l}{1+\lvert \ubar \rvert^p}, \mu \rangle = p^*_W$.
\end{lemma}
\begin{proof}
Observe that by coercivity of the cost, the mass of $\mu$ is bounded. Following
Alaoglu's theorem \cite[\S15.1]{Royden}, the unit ball in the vector space of compactly supported
measures is compact in the weak star topology. Therefore, any sequence of admissible solutions
for \eqref{eq:weakProblem} possesses a converging subsequence. Since this must be true for any
sequence, this is true for any minimizing sequence, which concludes the proof.
\end{proof}

Now, remark that \eqref{eq:weakProblem} can be seen as an instance of a conic program, called hereafter the
primal, in standard form (see for instance \cite{Barvinok2002convexity}):
\begin{equation}
\label{eq:primal}
\begin{array}{rcll}
p_W^* & = & \displaystyle \inf\limits_{x_p} & \displaystyle \langle  x_p, c \rangle_p  \\
&& \mathrm{s.t. }  & \mathcal{A} \, x_p = b, \\
&&& x_p \in  E^+_p
\end{array}
\end{equation}
with decision variable $x_p := \mu \in E_p := \SignedMeasures(T \times Y \times \gamma \R^m)$, and cost
$c := \frac{l}{1+\lvert \ubar \rvert^p} \in F_p := \ContFunc(T \times Y \times \gamma \R^m)$.
The notation $\langle  x_p, c \rangle_p$ refers to the duality between $E_p$ and $F_p$.
The cone $E^+_p$ is the non-negative orthant of $E_p$.
The linear operator $\mathcal{A}:E_p \rightarrow [\ContFunc^1(T \times Y)]'$ is the adjoint operator of
$\mathcal{A}': \ContFunc^1(T \times Y) \rightarrow \ContFunc(T \times Y \times \gamma \R^m)$ defined by
\begin{equation}
v \mapsto \mathcal{A}' v := \frac{\partial v}{\partial t} \frac{1}{1+\lvert \ubar \rvert^p}+  \frac{\partial v}{\partial \ybar} \frac{f(t,\ybar,\ubar)}{1+\lvert \ubar \rvert^p}.
\end{equation}
The right hand side is $b := \delta_{(t_f,y(t_f))}(\diff t \diff \ybar) - \delta_{(t_0,y(t_0))}(\diff t \diff \ybar)   \in E_d :=  [\ContFunc^1(T \times Y)]' $.

\begin{lemma}
The conic dual of \eqref{eq:weakProblem} is given by
\begin{equation}
\label{eq:dualExplicit}
\begin{array}{rcll}
d^* & = & \sup\limits_{v} & v(t_f,y(t_f)) - v(t_0,y(t_0))  \\
&& \mathrm{s.t.}  & \frac{l(t,\ybar,\ubar)}{1+\lvert \ubar \rvert^p} - \mathcal{A}' v(t,\ybar) \geq 0, \:\:\forall \: (t,\ybar,\ubar) \in  T \times Y \times \gamma \R^m, \\
&&& v \in \ContFunc^1(T \times Y).
\end{array}
\end{equation}
\end{lemma}

\begin{proof}
Following standard results of conic duality (see \cite{Anderson} or \cite{Barvinok2002convexity}), the conic dual of \eqref{eq:primal} is given by
\begin{equation}
\label{dual}
\begin{array}{rcll}
d^* & = & \sup\limits_{x_d} & \langle b, x_d \rangle_d  \\
&& \mathrm{s.t.}  & c-\mathcal{A}' x_d \in F_p^+, \\
&&& x_d \in E_d
\end{array}
\end{equation}
where decision variable $x_d :=  v \in F_d := \ContFunc^1(T \times Y)$ and (pre-)dual cone $F_p^+$ is the positive orthant of $F_p$.
The notation $\langle  b, x_d \rangle_d$ refers to the duality between $E_d$ and $F_d$.
The lemma just details this dual problem.
\end{proof}

Once the duality relationship established, the question arises of whether a duality gap may 
occur between linear problems \eqref{eq:dualExplicit} and \eqref{eq:weakProblem}. The following theorem discards such a
possibility:
\begin{theorem}
\label{th:noGap}
There is no duality gap between \eqref{eq:dualExplicit} and \eqref{eq:weakProblem}: if there is an admissible
vector for \eqref{eq:dualExplicit}, then
\begin{equation}
p_W^* = d^*.
\end{equation}
\end{theorem}
\begin{proof}
Following \cite[Th. 3.10]{Anderson} (see also the exposition in \cite[\S~4]{Barvinok2002convexity}), it is enough to show
that the
weak-star closure of the cone $C:=\left\lbrace (\langle x_p, c \rangle_p , \mathcal{A} \, x_p) : \; x_p \in E^+_p \right\rbrace$
belongs to $\R \times E_d$.

To prove closure, one may show that, for any sequence of admissible solutions $(x^k_p)_{k \in \N}$, all accumulation points of
$(\langle x^k_p, c \rangle_p , \mathcal{A} \, x^k_p)_{k \in \N}$ belong to $C$. Note that Lem.~\ref{th:solAttained} establishes that any sequence $(x^k_p)_{k \in \N}$
has a converging subsequence. Therefore, all that is left to show is the weak-star continuity of $\mathcal{A}$. Following \cite{sicon}, this can be shown by noticing that
$\mathcal{A}'$ is continuous for the strong topology of $\ContFunc^1(T \times Y)$, hence for its associated weak topologies.
Operators $\mathcal{A}$ is therefore weakly-star continuous, and each sequence $(\langle x^k_p, c \rangle_p , \mathcal{A} x^k)_{k \in \N}$
converges in $C$, which concludes the proof.
\end{proof}

Note that what is not asserted in Theorem~\ref{th:noGap} is the existence of a continuously differentiable function for which
the optimal cost is attained in dual problem \eqref{eq:dualExplicit}. Indeed, it is a well-known fact that
value functions of optimal control problems, to which $v$ is closely related, may not be
continuous, let alone continuously differentiable. However, there does exist an admissible
vector of measures for which the optimal cost of primal \eqref{eq:weakProblem} is attained (whenever there
exists an admissible solution), following Lem.~\ref{th:solAttained}. This furnishes practical motivations for
approaching the problem via its primal on measures, as the dual problem will be solved
anyway as a side product, see the later sections of this article.

\begin{corollary}
\label{th:CNSglobal}
Let $\mu$ be admissible for \eqref{eq:weakProblem}. Then, there exists a sequence $(v^k)_{k \in \N} \in \ContFunc^1(T \times Y )$, with each element admissible for \eqref{eq:dualExplicit} and such that
\begin{equation}
\lim_{k \rightarrow \infty} \langle \frac{l(t,\ybar,\ubar)}{1+\lvert \ubar \rvert^p} - \mathcal{A} v^k, \mu \rangle  = 0,
\end{equation}
if and only if $\mu$ is a solution of \eqref{eq:weakProblem}.
\end{corollary}
\begin{proof}
This corollary exploits weak duality via the complementarity condition
\[
\lim_{k \rightarrow \infty} \langle x^*_p, c-\mathcal{A}' x^k_d \rangle_p = 0
\]
if $x^*_p$ is optimal for the primal and $x^k_d$ is minimizing for the dual, and exploits Lem.~\ref{th:solAttained} guaranteeing the existence of an optimal $x^*_p$.
\end{proof}

This last corollary implies easy sufficient conditions for global optimality of the original and strong problem:
\begin{corollary}
Let $u(t)$ be admissible for \eqref{eq:OCP}, resp. $(\sigma, \nu)$ be admissible for \eqref{eq:strongProblem}. Let $\mu$ be their corresponding occupation measure. If there exists a sequence $(v^k)_{k \in \N} \in \ContFunc^1(T \times Y )$ satisfying the conditions of Cor.~\ref{th:CNSglobal}, then $u(t)$ is globally optimal for \eqref{eq:OCP}, resp. $(\sigma, \nu)$ is globally optimal for \eqref{eq:strongProblem}.
\end{corollary}

\section{Semidefinite relaxations}
\label{sec:momentRelax}

This section outlines the numerical method to solve \eqref{eq:weakProblem} in practice, by means of primal-dual
moment-SOS relaxations.
First, the problem is transformed to avoid explicitly handling $\gamma \R^m$.
Then, when problem data is restricted to be polynomial\footnote{This class of functions can cover a lot of cases in practice, see for instance the example in Section~\ref{sec:exampleSmeared}.}, occupation measures can be equivalently manipulated by their moments for the problem at hand. Then, this new infinite-dimensional problem is truncated so as to obtain a convex, finite-dimensional relaxation (see Section~\ref{solvelp}). Finally, we present a simple approach to reconstruct approximate trajectories from moment data, if more than the globally optimal cost is needed.

In the remainder of the paper, we make the following standing assumptions:
\begin{assumption}
\label{th:polynomialDataAssumption}
All functions are polynomial in their arguments:
\begin{equation}
   l, f \in \R[t,\ybar,\ubar].
\end{equation}
In addition,
set $Y$ is basic semi-algebraic, that is, defined by  finitely many polynomial inequalities
\begin{gather}
\label{eq:setKexplicit}
 Y = \left\lbrace \ybar \in \R^n : \, g_i(\ybar) \geq 0, i = 1,\ldots,n_Y   \right\rbrace
\end{gather}
with $g_i \in \R[\ybar]$, $ i = 1,\ldots,n_Y$. 
In addition, it is assumed that one of the $g_i$ enforces a ball constraint on all variables, which is possible w.l.g. since
$Y$ is assumed compact (see \cite[Th.~2.15]{lasserre} for slightly weaker conditions).
\end{assumption}

\subsection{Equivalent problem on compact sets}
\label{sec:compactsets}

Set $\gamma \R^m$ prevents the definition of moments of the form $\langle \ubar^\alpha, \mu \rangle$, since not all such polynomials are members of $\mathcal{R}$. In addition, manipulating very large quantities is bound to pose numerical issues. To overcome this difficulty, we use the fact that $\gamma \R^m$ is homeomorphic to the unit ball or a simplex in $\R^m$, see \cite{Kruzik1998optimization}. We leverage this necessary transformation to also remove the rational dependence in the control in \eqref{eq:weakProblem}.

For precisely this last practical reason, we consider specifically the $p$-norm when expressing $ \lvert \ubar \rvert $ as well as in the definition of $B^m$, the unit ball in $\R^m$. Define the mapping
\begin{equation}\label{eq:u2w}
\begin{array}{rclcl}
q & : & \gamma\R^m & \to & B^m \\
&& \ubar & \mapsto & \displaystyle \frac{\ubar}{(1+ \lvert \ubar \rvert^p)^{\frac{1}{p}}}
\end{array}
\end{equation}
with inverse
\begin{equation}\label{eq:w2u}
\begin{array}{rclcl}
q^{-1} & : &B^m  & \to & \gamma\R^m\\
&& \wbar & \mapsto &  \displaystyle \frac{\wbar}{(1- \lvert \wbar \rvert^p)^{\frac{1}{p}}}.
\end{array}
\end{equation}
As both $q$ and $q^{-1}$ are continuous, $q$ defines a proper homeomorphism.
Define now the pushforward measure
\begin{equation}
\label{eq:pushForward}
 \gamma(A \times B \times C) := \mu(A \times B \times q^{-1}(C))
\end{equation}
for every Borel subsets $A$, $B$, $C$ respectively in $T$, $Y$, $B^m$.

Finally, introduce the additional lifting variable
\[
\wbar_{0} := \left( 1- \sum_{i=1}^m \lvert\wbar_i \rvert^p \right)^{1/p}.
\]
This variable is easily constrained algebraically if, for instance, $p$ is even, since constraints $\wbar_{0} \geq 0$ and $\wbar_{0}^p = 1 - \sum_{i=1}^m \wbar_i^p$ uniquely determine $\wbar_{0}$. For clarity of exposition, we assume for the rest of the section that $p$ is even, while common alternative cases are simply worked out in some examples in Section \ref{sec:examples}.

With these hypotheses in mind, notice that for multi-index $\alpha \in \N^m$ such that $\lvert \alpha \rvert \leq p$,
\begin{equation}
\langle \frac{\ubar^\alpha}{1 + \lvert \ubar \rvert_p^p } , \mu \rangle = \langle \wbar_0^{p - \lvert \alpha \rvert} \wbar^\alpha, \gamma \rangle.
\end{equation}
That is, by homogenizing (with respect to the control) polynomials via the new variable $\wbar_0$, and using push-forward measure \eqref{eq:pushForward}, problem \eqref{eq:weakProblem} may be reformulated as:
\begin{equation}
\label{eq:weakProblemForward}
\begin{aligned}
p_W^* = \inf_{\gamma} \;\; & \langle \hat{l}(t,\ybar,\wbar), \gamma \rangle  \\
\text{s.t.} \;\; & \langle \frac{\partial v}{\partial t} \wbar_0^p  + \frac{\partial v}{\partial \ybar} \hat{f}, \gamma \rangle 
= v(t_f,y(t_f)) - v(t_0,y(t_0)) \quad \forall v \in \ContFunc^1(T \times Y)\\
& \gamma \in \PositiveMeasures(G)
\end{aligned}
\end{equation}
after defining
\[
G := T \times Y \times W, \quad W:= \left\lbrace (\wbar_0, \wbar) \in [0,1] \times B^m: \; \wbar_0 \geq 0, \wbar_{0}^p = 1 - \sum_{i=1}^m \wbar_i^p \right\rbrace
\]
and where $\hat{l}$ and $\hat{f}$ are the homogenizations by $\wbar_0$ with respect to the control of each term of resp. $l$ and $f$. That is, we multiply each of these terms by an appropriate power of $\wbar_0$ so that their degree with respect to the extended controls $(\wbar_0, \wbar)$ is equal to $p$.

Notice that \eqref{eq:weakProblemForward} is a measure LP with purely polynomial data. Indeed, it is easy to see that the support $G$ of $\gamma$ can be described as a semi-algebraic set:
\begin{equation}\label{eq:g}
 G = \left\lbrace  (t,\ybar,\wbar) \in \R^{1+n+1+m} : \; g_i(t,\ybar,\wbar) \geq 0, \; i = 1,\ldots , n_G  \right\rbrace.
\end{equation}
For ease of exposition (in the definition of moment matrices below), we also enrich the algebraic description of $G$ with the trivial inequality $g_0(t,y,w) := 1 \geq 0$. 
In addition, again for ease of exposition, we formally consider equality constraints as two inequality constraints.
Finally, without loss of generality since $G$ is bounded, we make the following
\begin{assumption}\label{compact}
One of the polynomials in the definition of set $G$ in (\ref{eq:g}) has the form $g_i(t,y,w):=r-\|(t,y,w)\|^2_2$ for
a sufficiently large constant $r$.
\end{assumption}

\subsection{Equivalent moment problem}
\label{sec:momentLP}

This subsection details how measure LP \eqref{eq:weakProblemForward} can be solved by an appropriate hierarchy of semidefinite programming (SDP)
relaxations, since it possesses polynomial data.

We first rewrite \eqref{eq:weakProblemForward} in terms of moments. Recall that the moment of degree $\alpha \in N^n$
of a measure $\mu(\diff x)$ on $X \subset \R^n$ is the real number\footnote{Most references on the subject use $y$ instead of $z$ for moments,
but we reserve this symbol for trajectories in this paper.}
\begin{equation}\label{moments}
z_\alpha = \langle x^\alpha , \mu \rangle.
\end{equation}
Then, given a sequence $z=(z_\alpha)_{\alpha \in \N^n}$, let $\ell_z:\R[x]\to\R$ be the linear functional
\begin{equation}
 f(x)= \sum_\alpha f_\alpha x^\alpha \in\R[x]
  \quad \mapsto \quad
  \ell_z(f)\,=\,\sum_\alpha f_\alpha z_\alpha. 
\end{equation}
Define the localizing matrix of order $d$ associated with sequence $z$ and
polynomial $g(x) = \sum_\gamma g_\gamma x^\gamma \in \R[x]$ as the real symmetric matrix $M_d(g\,y)$ whose entry $(\alpha,\beta)$ reads
\begin{align}
  \label{eq:locMomMat}
 [M_d(g\,z)]_{\alpha, \beta} & = \ell_z \left( g(x) \, x^{\alpha+\beta} \right) \\
                    & = \sum_\gamma g_\gamma \, z_{\alpha+\beta+\gamma},
  \quad \forall \alpha,\beta \in \N^n_d.
\end{align}
In the particular case that $g(x)=1$, the localizing matrix is called the moment matrix.
As a last definition, a sequence of reals $z=(z_{\alpha})_{\alpha \in \N^n}$ is said to have a representing measure if there exists a finite Borel measure $\mu$ on $X$, such that relation (\ref{moments}) holds for every $\alpha \in \N^n$.

The construction of the moment problem associated with \eqref{eq:weakProblemForward} can now be stated. Its decision variable is the sequence of moments $(z_{\beta})_{\beta \in \N^{1+n+1+m}}$ of measure $\gamma$, where each moment is given by $z_\beta := \langle e_\beta , \gamma \rangle, \beta \in \N^{1+n+1+m}$, where $e_\beta$ are polynomials of $\R[t,\ybar,\wbar]$ defining its monomial basis\footnote{This choice purely for ease of exposition, for numerical
reasons other bases may be more appropriate, e.g. Chebyshev polynomials.}:
\begin{equation}
\label{eq:momentBaseSpecified}
e_\beta := t^{\beta_1} \ybar_1^{\beta_2} \cdots \ybar_{n}^{\beta_{n+1}} \wbar_{0}^{\beta_{n+2}}
\wbar_{1}^{\beta_{n+3}} \cdots \wbar_{m}^{\beta_{1+n+1+m}},
\:\:\beta \in \N^{1+n+1+m}.
\end{equation}
As each term of the cost of \eqref{eq:weakProblemForward} is polynomial by assumption, they can be rewritten as
\begin{equation}
\sum_{\beta \in \N^{n+1+m+1}} c_{\beta} z_{\beta} = \ell_{z}(\hat{l}).
\end{equation}
That is, vector $(c_{\beta})$ contains the coefficients of polynomial $\hat{l}$ expressed in monomial basis \eqref{eq:momentBaseSpecified}.

Similarly, the weak dynamic constraints of \eqref{eq:weakProblemForward} need only be satisfied for countably many polynomial test functions $v \in \R[t,\ybar]$, since the measure $\gamma$ is supported on a compact subset of $\R^{1+n+1+m}$.
Therefore, for the particular choice of test function 
\begin{equation}
\label{eq:betaTestFun}
v_\alpha(t,\ybar) := t^{\alpha_1} \ybar_1^{\alpha_2} \cdots \ybar_n^{\alpha_{n+1}}, \:\:\alpha \in \N^{1+n},
\end{equation}
the weak dynamics define linear constraints between moments of the form
\begin{equation}
\label{eq:bbeta}
\sum_{\beta \in \N^{1+n+1+m}} a_{\alpha,\beta} z_{\beta}  = b_\alpha := v_\alpha(t_t,y(t_f)) - v_\alpha(t_0,y(t_0))
\end{equation}
where coefficients $a_{\alpha,\beta}$ can  be deduced by identification with
\begin{equation}
\label{eq:momLinConstr}
\sum_{\beta \in \N^{1+n+1+m}} a_{\alpha,\beta} z_{\beta} = \ell_{z}(\frac{\partial v_\alpha}{\partial t} \wbar_0^p  + \frac{\partial v_\alpha}{\partial \ybar} \hat{f}(t,\ybar,\wbar)), \:\:\alpha \in \N^{1+n}.
\end{equation}

Finally, the only nonlinear constraints are the convex semidefinite constraints for measure representativeness.
Indeed, it follows from Putinar's theorem \cite[Theorem 3.8]{lasserre} that 
the sequence of moments $z$ has a representing measure defined on
a set $G$ satisfying Assumption \ref{compact} if and only if $M_d(g_i\,z) \succeq 0$ for all $d \in \N$ and for all polynomials $g_i$ defining the set, where $\succeq 0$ means positive semidefinite.

This leads to the problem:
\begin{equation}
\label{eq:momentpb}
\begin{array}{rcll}
p^*_\infty & = & \displaystyle \inf_z &  \displaystyle \sum_{\beta \in \N^{1+n+1+m}} c_{\beta} z_{\beta}\\
&& \mathrm{s.t.} &  \displaystyle \sum_{\beta \in \N^{1+n+1+m}} a_{\alpha,\beta} z_{\beta}  = b_\alpha, \:\alpha \in \N^{1+n},\\
&&& M_d(g_i\, z) \succeq 0, \: i=0,1 ,\ldots , n_G, \: d \in \N.
\end{array}
\end{equation}

\begin{theorem}\label{th1}
Measure LP \eqref{eq:weakProblemForward} and infinite-dimensional SDP \eqref{eq:momentpb}
share the same optimum:
\begin{equation}
    p_W^* = p^*_\infty.
\end{equation}

\end{theorem}

For the rest of the paper, we will therefore use $p_W^*$ to denote the cost of measure LP \eqref{eq:weakProblemForward} or
infinite-dimensional SDP \eqref{eq:momentpb} indifferently.

\subsection{Moment hierarchy}
\label{solvelp}

The final step to reach a tractable problem is relatively obvious: infinite-dimensional SDP problem \eqref{eq:momentpb} is truncated to its first few moments.

To streamline exposition, first notice in \eqref{eq:momentpb} that $M_{d+1}( \cdot) \succeq 0$ implies $M_{d}( \cdot) \succeq 0$, such that when truncated, only the semidefinite constraints of highest order must be taken into account.
Now, let $d_0 \in \N$ be the smallest integer such that all criterion, dynamics and constraint monomials belong to $\N^{1+n+1+m}_{2 d_0}$.  This is the degree of the so-called first relaxation. For any relaxation order $d \geq d_0$, the decision variable is now the finite-dimensional vector $( z_{\alpha})_\alpha$ with $\alpha \in \N_{2 d}^{1+n+1+m}$, made of the first $ \left( \begin{smallmatrix} 1+n+1+m+2d \\ 1+n+1+m \end{smallmatrix} \right)$ moments of  measure $\gamma$.  Then, define the index set
\[
  \bar{\N}^{1+n}_{2d} := \left\lbrace \alpha \in \N^{1+n}: \, \deg \left( \frac{\partial v_\alpha}{\partial t} \wbar_0^p  + \frac{\partial v_\alpha}{\partial \ybar} \hat{f}(t,\ybar,\wbar) \right) \leq 2 d, \right\rbrace
\]
viz. the set of monomials for which test functions of the form \eqref{eq:betaTestFun} lead to linear constraints of appropriate degree. By assumption, this set is finite and not empty -- the constant monomial always being a member.

Then, the SDP relaxation of order $d$ is given by
\begin{equation}\label{eq:momentrelax}
\begin{array}{rcll}
p^*_d & = &  \inf\limits_z &  \displaystyle  \sum_{\beta \in \N^{1+n+1+m}_{2d}} c_{\beta} z_{\beta} \\
&& \text{s.t.} &  \displaystyle \sum_{\beta \in \N^{1+n+1+m}_{2d}} a_{\alpha,\beta} z_{\beta}  = b_\alpha, \:\alpha \in \bar{\N}^{1+n}_{2d}, \\
&&& M_d(g_i\, z) \succeq 0, \: i=0,1 ,\ldots , n_G.
\end{array}
\end{equation}
Notice that for each relaxation, we obtain a standard  finite-dimensional SDP
that can be solved numerically by off-the-shelf software. In addition, the relaxations
converge asymptotically to the cost of the moment LP:

\begin{theorem}\label{th2}
\begin{equation}
p^*_{d_0} \leq p^*_{d_0+1} \leq \cdots \leq p^*_{\infty}  = p_W^*.
\end{equation}
\end{theorem}

\begin{proof}
By construction, observe that $j > i$ implies $p^*_{d_0+j} \geq p^*_{d_0+i} $, viz.
the sequence $p^*_d$ is monotonically non-decreasing. Asymptotic convergence to $p_W^*$
follows from \cite[Theorem 3.8]{lasserre}.
\end{proof}

Therefore, by solving the truncated problem for ever greater relaxation orders, we will obtain a monotonically non-decreasing sequence
of lower bounds to the true optimal cost.

\subsection{SOS hierarchy}
\label{sec:SOS}

As for the measure LP of Section \ref{sec:measureLP}, the moment relaxations detailed in the previous section possess a conic dual. In this section, we show that this dual problem can be interpreted as a polynomial sum-of-squares (SOS) strengthening of the dual outlined in Section \ref{sec:duality}.
The exact form of the dual problem is an essential aspect of the numerical method, since it will be solved implicitly whenever primal-dual interior point methods are used for the moment hierarchy of Section \ref{solvelp}.

Let ${\mathbb S}^n$ be the space of symmetric $n \times n$ real matrices equipped with the
inner product $\langle A, B \rangle := \mathrm{trace} \, AB$ for $A,B \in {\mathbb S}^n$.
In problem \eqref{eq:momentrelax}, let us define the matrices $A_{i,\beta} \in {\mathbb S}^{\left( \begin{smallmatrix} 1+n+1+m+d \\ 1+n+1+m \end{smallmatrix} \right)}$ satisfying the identity
\begin{equation}
M_d(g_i\:z) = \sum_\beta A_{i,\beta} z_{\beta}
\end{equation}
for every sequence $(z_{\beta})_{\beta}$ and $i=0,1,\ldots,n_G$.

\begin{proposition}
The conic dual of moment SDP \eqref{eq:momentrelax} is given by the SOS SDP
\begin{equation}\label{eq:dualmomentrelax}
\begin{array}{ll}
\sup\limits_{x, X} &  \displaystyle \sum_{\alpha \in \bar{\N}^{1+n}_{2d}} b_{\alpha} x_{\alpha}  \\
\text{s.t.} &  \displaystyle \sum_{\alpha \in \bar{\N}^{1+n}_{2d}} a_{\alpha,\beta} x_\alpha  + \sum_{i=0}^{n_G} \langle A_{i,\beta}, X_{i} \rangle =  c_{\beta}, \: \beta \in \N^{1+n+1+m}_{2 d},\\
& X_i \in {\mathbb S}^{\left( \begin{smallmatrix} 1+n+1+m+d \\ 1+n+1+m \end{smallmatrix} \right)}, \quad X_{i} \succeq 0,  \quad i = 0, 1,\ldots , n_G, \\
& x_\alpha \in \R, \quad \alpha \in \bar{\N}^{1+n}_{2d}.
\end{array}
\end{equation}

\end{proposition}

\begin{proof}
Replacing the equality constraints in \eqref{eq:momentrelax} as two inequalities, moment relaxation \eqref{eq:momentrelax} can be written as an instance of linear program \eqref{dual}, whose dual is given symbolically by \eqref{eq:primal}. Working out the details leads to the desired result, using for semidefinite constraints the duality bracket as explained earlier.
\end{proof}

The relationship between \eqref{eq:dualmomentrelax} and \eqref{eq:dualExplicit} might not be obvious at a first glance. Denote by $\Sigma[x]$ the subset of $\R[x]$ that can be expressed as a finite sum of squares of polynomials in the variable $x$. Then a standard interpretation (see e.g. \cite{lasserre}) of 
\eqref{eq:dualmomentrelax} in terms of such objects is given by the next proposition.

\begin{proposition} \label{th:polyStrength}
Semidefinite problem  \eqref{eq:dualmomentrelax} can be stated as the following polynomial SOS
strengthening of problem \eqref{eq:dualExplicit}:
\begin{equation}\label{eq:SOSstrength}
\begin{array}{ll}
\sup &  v(t_f,y(t_f)) - v(t_0,y(t_0)) \\
\text{s.t.} & \displaystyle \hat{l} - \frac{\partial v}{\partial t} \wbar_0^p  - \frac{\partial v}{\partial \ybar} \hat{f}  = \sum_{i=0}^{n_G} g_i s_{i}, \\
\end{array}
\end{equation}
where the maximization is w.r.t. the vector of coefficients $x$ of polynomial
\[
v(t,\ybar) = \sum_{\alpha \in \bar{\N}^{n+1}_{2d}} x_\alpha v_\alpha(t,\ybar),
\]
and the vectors of coefficients of polynomials
\[
s_{i} \in \Sigma [t,\ybar,\wbar], \quad \deg g_i \, s_{i} \leq 2d,  \quad i = 0, \ldots , n_G.
\]
\end{proposition}

\begin{proof}
By \eqref{eq:bbeta}, the cost of \eqref{eq:SOSstrength} is equivalent to \eqref{eq:dualmomentrelax}. Then multiply each scalar constraint (indexed by $\beta$) by monomial $e_\beta$ in definition \eqref{eq:momentBaseSpecified}, and sum them up. By definition, $\sum_\beta c_{\beta} e_\beta = \hat{l}$. Notice also that $\sum_{\beta} \sum_{\alpha} a_{\alpha,\beta} x_\alpha e_\beta = \frac{\partial v}{\partial t} \wbar_0^p  + \frac{\partial v}{\partial \ybar} \hat{f}$. The conversion from the semidefinite terms to the SOS exploits their well-known relationship (e.g. \cite[\S{}4.2]{lasserre}) to obtain the desired result, by definition \eqref{eq:locMomMat} of the localizing matrices.
\end{proof}

Prop. \ref{th:polyStrength} specifies in which sense ``polynomial SOS strengthenings'' must be understood: positivity constraints of problem\footnote{More exactly, the dual of \eqref{eq:weakProblemForward}, which is straightforward to explicit from \eqref{eq:dualExplicit} and map \eqref{eq:u2w}.} \eqref{eq:dualExplicit} are enforced by SOS certificates, and the decision variable of \eqref{eq:dualExplicit} is now limited to polynomials of appropriate degrees.
Finally, the following result states that no numerical troubles
are expected when using classical interior-point algorithms
on the primal-dual semidefinite pair (\ref{eq:momentrelax}-\ref{eq:dualmomentrelax}).

\begin{proposition}
The infimum in primal  problem \eqref{eq:momentrelax} is equal to
the supremum in dual problem \eqref{eq:dualmomentrelax},
i.e. there is no duality gap.
\end{proposition}

\begin{proof}
By Assumption \ref{th:polynomialDataAssumption}, one of the polynomials $g_i$ in the description of set $Y$
enforces a ball constraint, and so does the ball constraint on $W$ and the representation of the intervals.
The corresponding localizing constraints $M_d(g_i\, z) \succeq 0$ then implies
that the vector of moments $z$ is bounded in semidefinite problem \eqref{eq:momentrelax}. 
Then to prove the absence of duality gap, we use
the same arguments as in the proof of Theorem 4 in Appendix D of \cite{roa}.
\end{proof}

\subsection{Approximate solution reconstruction}
\label{sec:reconstruction}

Solving the hierarchy of primal/dual problems \eqref{eq:momentrelax}/\eqref{eq:dualmomentrelax} simply yields a monotonically  non-decreasing and converging sequence of lower bounds on the relaxed cost $p_W^*$. If actual trajectories and/or controls must also be reconstructed, additional computations must be carried out, since this information is encoded within the moments and the dual polynomial certificates.

We follow the simple strategy proposed in \cite{claeys2014reconstruction}, which is well-suited to recover discontinuous trajectories. In this framework, a grid $G_\epsilon$ of $G$ is fixed, such that set $G_\epsilon$ has finite cardinality and $\mathrm{dist} (z, G_\epsilon) \leq \epsilon $ for all $z \in G$. Then, one seeks to minimize the distance between the truncated moment sequence  of a measure supported on $G_\epsilon$ , and the moments returned by any of the relaxations. The supremum norm is particularly interesting, since it corresponds to the truncation of the weak star norm on measures. Because $G_\epsilon$ has finite cardinality, this optimization problem is a simple finite-dimensional linear program to solve with the atomic weights as decision variables. Obviously, as $\epsilon$ is decreased and the relaxation order is increased, those atomic measures approximate closely the moments of the global optimal occupation measure (assuming its uniqueness), and good approximation of controls and trajectories can be recovered. If high precision is required, those may later be refined by applying any local optimization technique, or by simply working out necessary optimality conditions based on the observed solution structure. Finally, note that to circumvent the exponential growth in the dimension of $G$ of the computational method entails, one can simply identify one coordinate (state or control) at a time, by simply considering moments of time and that coordinate. Then, one only needs to grid the projection of $G$ on a $2$-dimensional grid, which is well within reach of current mid-scale LP solvers.

\section{Extensions}
\label{sec:extensions}

The approach can easily take into account a free initial state and/or time, by introducing an initial occupation measure $\mu_0 \in \PositiveMeasures(\{t_0\}\times Y_0)$, with $Y_0 \subset \R^n$
a given compact set of allowed initial conditions, such that for an admissible starting point $(t_0,y(t_0))$,
\begin{equation}
\label{eq:weakInitial}
\langle v, \mu_0 \rangle = v(t_0,y(t_0)),
\end{equation}
for all continuous test functions $v(t,\ybar)$ of time and space. That is, in weak problem \eqref{eq:weakProblemForward}, one replaces some of the boundary conditions by injecting \eqref{eq:weakInitial} and making $\mu_0$ an additional decision variable. In dual \eqref{eq:dualExplicit}, this adds a constraint on the initial value of $v$ and modifies the cost.

Similarly, a terminal occupation measure $\mu_f$ can be introduced for free terminal states and/or time. Note that injecting  both \eqref{eq:weakInitial} and its terminal counterpart in \eqref{eq:weakProblemForward} requires the introduction of an additional affine constraint to exclude trivial solutions, e.g. $\langle 1, \mu_0 \rangle = 1$ so that $\mu_0$ is a probability measure.
In dual problem \eqref{eq:dualExplicit}, this introduces an additional decision variable.

More interestingly, the initial time may be fixed w.l.g. to $t_0 = 0$, but only the spatial probabilistic distribution of initial states is known. Let $\xi  \in \PositiveMeasures(Y_0)$ be the probability measure whose law describes such a distribution. Then, the additional constraints for \eqref{eq:weakProblemForward} are
\begin{equation}
\langle v, \mu_0 \rangle = \int_{Y_0} \!\! v(0,\ybar) \, \xi(\diff \ybar).
\end{equation}
As remarked in \cite{sicon}, this changes the interpretation of LP \eqref{eq:weakProblemForward} to the minimization of the expected value of the cost given the initial distribution. See also \cite{roa} for the Liouville interpretation of the LP as transporting the probability $\xi$ along the optimal flow.

Finally, integral constraints can be easily taken into account in our framework. For instance, if $\delta = 0$ in Assumption~\ref{th:assGrowth}, it is essential to add a constraint of the type $\int_T \lvert u(t) \rvert^p \, \diff t \leq u_{\max}$ in \eqref{eq:OCP}, with $u_{\max}$ a given bound on the norm of the control. This results in problem \eqref{eq:weakProblem} in additional linear constraint $\langle \frac{ \lvert \ubar \rvert^p}{1 + \lvert \ubar \rvert^p}  , \mu \rangle \leq u_{\max}$, hence in an additional scalar decision variable in the various dual problems.

\section{Examples}
\label{sec:examples}

This section presents several examples of the approach. Practical construction of the semidefinite relaxations were implemented via the \texttt{GloptiPoly} toolbox \cite{GloptiPoly}, and solved numerically via \texttt{SeDuMi} \cite{SeDuMi}.

\subsection{Simple impulse}
\label{sec:exImpulse}
We consider optimal control problem \eqref{preface1}, our first motivating example. 
The unique solution is to apply an impulse of unit amplitude at $t=\frac{1}{2}$, such that minimizing sequences in $L^2$ tend weakly to the Dirac measure. Therefore, the optimal DiPerna-Majda measure is given by
\begin{equation}
\sigma(\diff t) = \diff t + \delta_{\frac{1}{2}}(\diff t), \qquad \nu (\diff\ubar|t) = \begin{cases} \delta_0(\diff \ubar) &\mbox{if } t \neq \frac{1}{2}, \\
\delta_{+ \infty}(\diff \ubar) & \mbox{if } t = \frac{1}{2} \end{cases}
\end{equation}
with optimal trajectory $y(t)=0$ if $0 \leq t \leq \frac{1}{2}$ and $y(t)=1$ if $1/2 < t \leq 1$.

The problem is rewritten as the measure LP \eqref{eq:weakProblemForward}, with\footnote{Obviously, we could make a simple substitution to make the problem affine in the control and use the more efficient relaxations of \cite{impulse}. This example is really to showcase the flexibility of our method. } $\wbar_0 := \sqrt{1 - \wbar_1^2}$ algebraically constrained in the description of our semi-algebraic measure support $G$ as $\wbar_0^2 = 1 - \wbar_1^2$ and $\wbar_0 \geq 0$:
\begin{equation}
\begin{aligned}
\inf_{\gamma,\mu_1} \;\; & \langle (t-\frac{1}{2})^2 \wbar_1^2, \gamma \rangle \\
\text{s.t.} \;\; & \langle \frac{\partial v}{\partial t} \wbar_0^2  + \frac{\partial v}{\partial \ybar} \wbar_1^2 , \gamma \rangle = \langle v(1,\cdot), \mu_1 \rangle - v(0,y(0)), \:\:\forall v \in \R[t,\ybar]\\
& \gamma \in \PositiveMeasures(G), \:\: \mu_1 \in \PositiveMeasures(\R).
\end{aligned}
\end{equation}
Recall in particular that $u = \infty$ if and only if $w_0 = 0$.
Solving the problem with \texttt{GloptiPoly} leads at the fifth relaxation to moments of the form
\begin{align*}
\left(\ell_z(t^k) \right)_{k \in \N} & = ( 2.0000,\,  1.0000,\,  0.5833,\,  0.3750,\,  0.2625,\,  0.1979,\, \ldots ), \\
\left(\ell_z(\wbar^k) \right)_{k \in \N} & = ( 2.0000,\,    1.0101,\,    1.0000,\,    0.9943,\, 0.9903,\,    0.9873,\, \ldots )
\end{align*}
to be compared with the moments of the known optimal solution:
\begin{align*}
\left( \langle t^k, \gamma^* \rangle \right)_{k \in \N} & = ( 2.0000,\,  1.0000,\,  0.5833,\,  0.3750,\,  0.2625,\,  0.1979,\, \ldots ), \\
\left( \langle \wbar^k, \gamma^* \rangle \right)_{k \in \N} & = ( 2.0000,\,    1.0000,\,    1.0000,\,    1.0000,\,    1.0000,\,   1.0000,\, \ldots ).
\end{align*}
Finally, Figure~\ref{fig:trajImpulse} shows the reconstructed trajectory using the method outlined in Section \ref{sec:reconstruction}. The solution agrees with the true solution, even during the jump.

\begin{figure}
\centering
\includegraphics[width=0.8\textwidth]{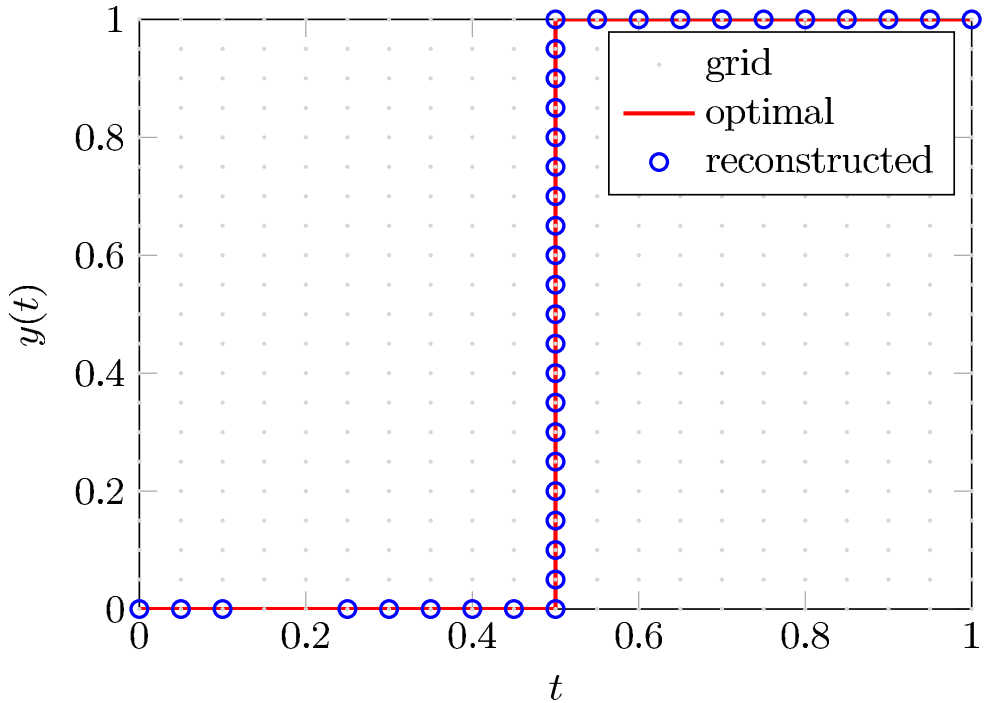}
\caption{Reconstruction trajectory for Example~\ref{sec:exImpulse}.}
\label{fig:trajImpulse}
\end{figure}

\subsection{Smeared impulses}
\label{sec:exampleSmeared}

We consider now optimal control problem \eqref{zadani3}, our second motivating example. This sequence tends weakly to the DiPerna-Majda measure given by
\begin{equation}
\sigma(\diff t) = 2 \, \diff t, \qquad \nu(\diff \ubar|t) = \frac{1}{2} \delta_0(\diff \ubar) + \frac{1}{2} \delta_{+\infty}(\diff \ubar),
\end{equation}
with optimal trajectory given by $y(t)=t$.

In this case, one can avoid the use of lift variable $\wbar_0$, since with the positivity constraint on $u$, we have simply $\wbar_0 = 1 - \wbar_1$. Measure LP \eqref{eq:weakProblemForward} is then expressed as
\begin{equation}
\label{eq:toto}
\begin{aligned}
\inf_{\gamma, \mu_1} \;\; & \langle \frac{(1-\wbar)^3 \, \wbar^2}{(1-\wbar)^4+\wbar^4} + (1-\wbar) \, (\ybar_1-t)^2, \gamma \rangle \\
\text{s.t.} \;\; &\langle \frac{\partial v}{\partial t} (1-\wbar) + \frac{\partial v}{\partial \ybar} \wbar, \gamma \rangle =  \langle v(1,\cdot), \mu_1 \rangle - v(0,y(0)), \:\: \forall v \in \R[t,\ybar]\\
& \gamma \in \PositiveMeasures([0,1]^3), \:\: \mu_1 \in \PositiveMeasures(\R)
\end{aligned}
\end{equation}
Because we started with a cost rational in the control, there remains a rational expression in \eqref{eq:toto} despite the change of variables introduced in Section \ref{sec:compactsets}.
This rational term can then treated by introducing the lifting variable $r = \frac{(1-\wbar)^3 \, \wbar^2}{(1-\wbar)^4+\wbar^4}$, constrained algebraically via the equation $r ((1-\wbar)^4+\wbar^4) = (1-\wbar)^3 \wbar^2$.

Solving the problem with GloptiPoly leads at the fourth relaxation leads to moments of the form
\begin{align*}
\left( \ell_z(t^k) \right)_{k \in \N} & = ( 2.0026  ,\,  1.0026  ,\,  0.6692  ,\,  0.5026  ,\,  0.4026  ,\,  0.3359 ,\, \ldots ), \\
\left( \ell_z(\wbar^k) \right)_{k \in \N} & = ( 2.0026  ,\,   1.0026  ,\,   1.0012  ,\,   0.9999  ,\,   0.9985  ,\,   0.9972,\, \ldots )
\end{align*}
to be compared with the moments of the known optimal solution
\begin{align*}
\left( \langle t^k, \gamma^* \rangle \right)_{k \in \N} & = ( 2.0000  ,\,  1.0000 ,\,  0.6667  ,\,  0.5000 ,\,   0.4000  ,\,  0.3333,\, \ldots ), \\
\left( \langle \wbar^k , \gamma^* \rangle \right)_{k \in \N} & = ( 2.0000,\,    1.0000,\,    1.0000,\,    1.0000,\,    1.0000,\,   1.0000,\, \ldots ).
\end{align*}
The trajectory reconstructed with the method of Section \ref{sec:reconstruction} is given in Fig.~\ref{fig:smearedreconst}. The reconstructed solution agrees with the true optimal solution.

\begin{figure}
\centering
\includegraphics[width=0.8\textwidth]{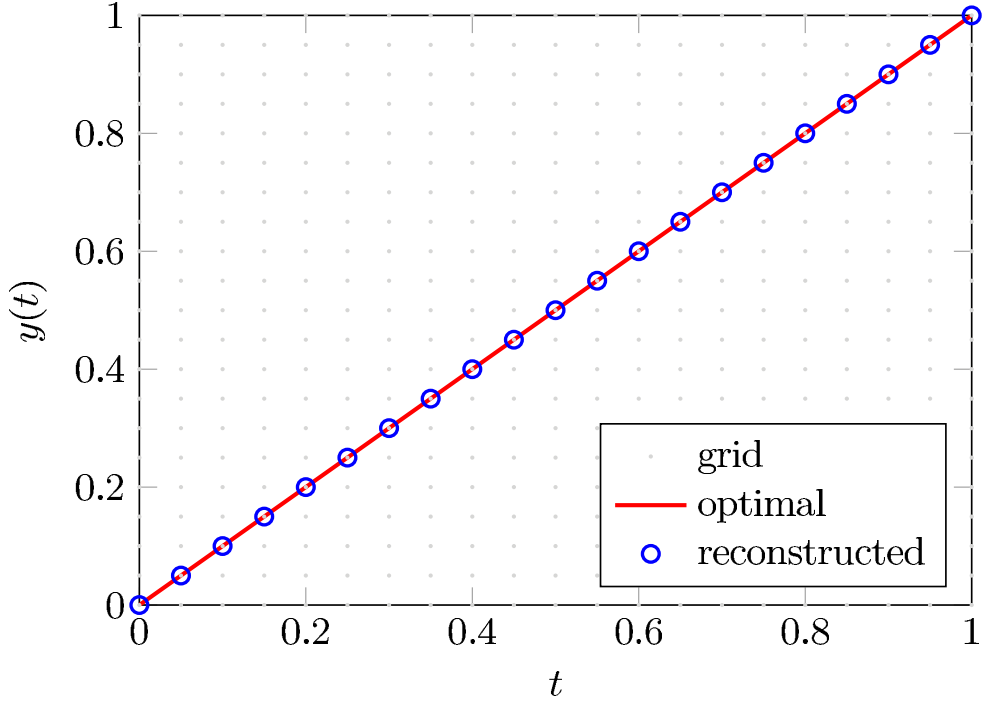}
\caption{Reconstruction trajectory for Example~\ref{sec:exampleSmeared}.}
\label{fig:smearedreconst}
\end{figure}

\subsection{Non-convex problem}
\label{sec:exRDV}

We consider now a simplified planar orbital rendezvous problem constrained in a non-convex domain:
\begin{equation}
\begin{aligned}
\inf_{u} \;\; & \int_{0}^{1} \lvert u \rvert \, \diff t \\
\text{s.t.} \;\; & \dot{y}_1 =  \pi \, y_2,  \:\: \dot{y}_2 =  -\pi \, y_1 + u,  \\
& y(0) =  (1/2,\:0), \:\: y(1) = (-1,\:0),\\
& y_1^2(t) + y_2^2(t) \leq 2, \:\: y_1^2(t) + (y_2(t)- 1/2)^2 \geq 1/4, \\
& u \in L^1([0,1]).
\end{aligned}
\end{equation}
Instead of using $\wbar_0 := 1- \lvert \wbar_1 \rvert$ as for even $p$ in Section \ref{sec:compactsets}, one could instead use the lifting variable $r = \lvert \wbar_1 \rvert$, algebraically represented as $r^2 = \wbar_1^2$, $r \geq 0$. Then LP problem \eqref{eq:weakProblemForward} is written as
\begin{equation}
\begin{aligned}
\inf_{\gamma} \;\; & \langle r, \gamma \rangle \\
\text{s.t.} \;\; &\langle \frac{\partial v}{\partial t} (1-r) + \frac{\partial v}{\partial \ybar_1}
 2 \pi \, \ybar_2 (1-r)  + \frac{\partial v}{\partial \ybar_2} (-2 \pi \, \ybar_1 (1-r) + \wbar_1),
 \gamma \rangle \\
& \quad\quad = v(1,y(1)) - v(0,y(0)), \:\: \forall v \in \R[t,\ybar]\\
& \gamma \in \PositiveMeasures(G)
\end{aligned}
\end{equation}
with semi-algebraic set
\[
\begin{array}{l}
G = \{(t,y_1,y_2,w,r) \in \R^5 : t(1-t) \geq 0, \: 2-y^2_1-y^2_2 \geq 0,\\
\quad\quad  -1/4+y^2_1 + (y^2_2-1/2)^2 \geq 0, \:
1-w^2_1 \geq 0, \: r \geq 0, \: r^2 = w^2_1\}
\end{array}
\]
defined accordingly to the constraints in the optimal control problem.
The trajectory reconstructed from moments of the 6th relaxation is reported back on Fig.~\ref{fig:RDVreconst}, with a relaxation cost of $p_6^*=0.824$. From this, it is easy to infer a candidate optimal trajectory involving an impulse of $\approx 0.273$ at $t=0$, followed by an impulse of $\approx -0.227$ at $t \approx 0.318$, followed by a control in feedback form of $u = y_1/(2y_2+1)$ to steer around the obstacle, followed by a free coasting arc. This admissible policy has a cost of $\approx 0.846$, which, given the relaxation cost, strongly suggests its global optimality.

\begin{figure}
\centering
\includegraphics[width=0.8\textwidth]{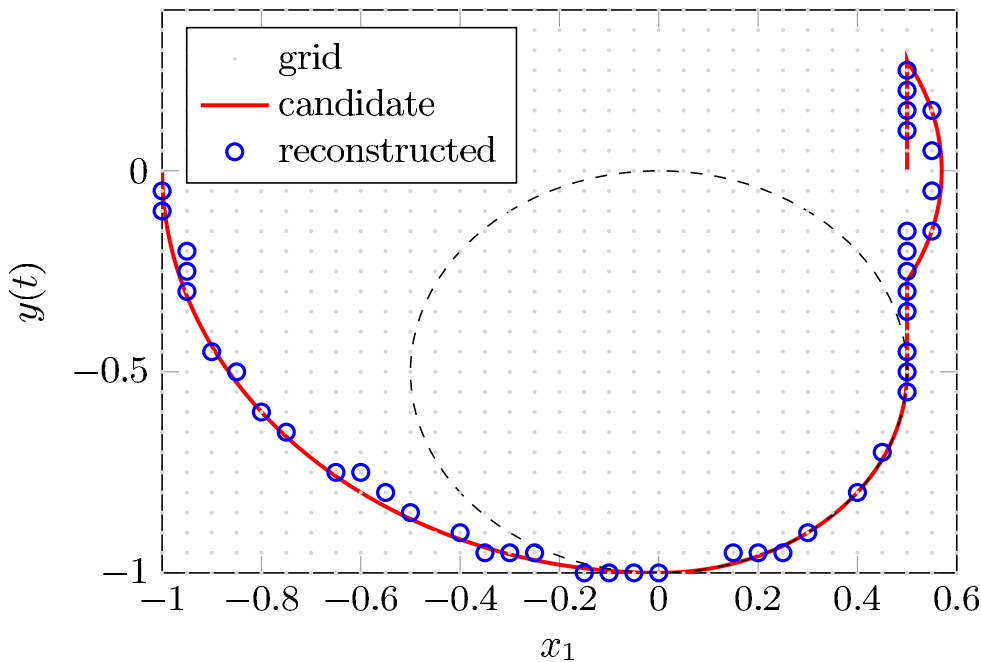}
\caption{Reconstruction trajectory for the example of Ex.~\ref{sec:exRDV}.}
\label{fig:RDVreconst}
\end{figure}

\subsection{Weierstrass' example}
\label{sec:weierstrass}

Consider the optimal control problem
\begin{equation}
\begin{aligned}
\inf_u \;\; & \int_{-1}^{1} t^2 u^2 \, \diff t \\
\text{s.t.} \;\; & \dot{y} =  u  \\
& y(-1) =  -1, \quad y(1) = 1 \\
& u \in L^2([-1,1]), \quad u(t) \geq 0.
\end{aligned}
\end{equation}
Note the lack of coercivity in the cost integrand for $t=0$, so that the standing Assumptions \ref{th:assGrowth} are not met.
Following \cite[Section 5.4]{Stein} a minimizing sequence for this problem is
\begin{equation}
\label{eq:weierstrassMinimizing}
   u^k(t) = \begin{cases} k & \mbox{if } \frac{-1}{k} \leq t \leq \frac{1}{k}, \\ 0 & \mbox{otherwise}. \end{cases}
\end{equation}
Note however that $\lim_{k \rightarrow \infty} \int_{-1}^{1} \! (u^k(t))^2 \, dt \rightarrow \infty$, such that the infimum cannot be attained by a DiPerna-Majda measure. 

Solving the LP problem \eqref{eq:weakProblemForward} with {\tt GloptiPoly} leads to numerical issues,
as the mass of $\gamma$ grows without bounds at each relaxation, as expected.
We are currently investigating suitable analytical and numerical frameworks to cope with this problem.
However, if one introduces an additional $L^2$ norm on the control, the problem becomes tractable by our approach.
The new optimal minimizing sequence will obviously be different from \eqref{eq:weierstrassMinimizing}.

\section*{Acknowledgment}

This research was supported by a project between the Academy of Sciences of the Czech Republic (AV\v{C}R) and the French Centre National de
la Recherche Scientifique (CNRS) entitled ``Semidefinite programming for nonconvex problems of calculus of variations and optimal control''.
The first author was also supported by the United Kingdom Engineering and Physical Sciences Research Council under Grant EP/G066477/1.


\begin{thebibliography}{XX}


\bibitem{Ambrosio2000BV}
L.~Ambrosio, N.~Fusco, D.~Pallara.
\newblock Functions of bounded variation and free discontinuity problems.
\newblock Oxford University Press, UK, 2000.

\bibitem{Anderson}
E. J. Anderson, P.~Nash.
\newblock Linear programming in infinite-dimensional spaces: theory and applications.
\newblock Wiley, 1987.

\bibitem{Barvinok2002convexity}
A.~Barvinok.
\newblock A course in convexity.
\newblock American Mathematical Society, Providence, NJ, 2002.

\bibitem{BenTal2001lecture}
A. Ben-Tal, A. Nemirovski.
\newblock Lectures on modern convex optimization: analysis, algorithms,
  and engineering applications.
\newblock SIAM, Philadelphia, 2001.

\bibitem{blaquire}
A.~Blaqui\`ere.
\newblock Impulsive optimal control with finite or infinite time horizon.
\newblock Journal of Optimization Theory and Applications, 46:431--439, 1985.

\bibitem{Bressan1988graphcompl}
A. Bressan, F. Rampazzo.
\newblock On differential systems with vector-valued impulsive controls.
\newblock Unione Matematica Italiana. Bollettino. B., 2(3):641--656, 1988.

\bibitem{bryson-ho}
A. E. Bryson, Y.C. Ho.
\newblock Applied optimal control theory.
\newblock Ginn \& Co., Waltham, 1969.

\bibitem{claeys2014thesis}
M.  Claeys.
\newblock Mesures d'occupation et relaxations semi-d\'efinies pour la
  commande optimale.
\newblock PhD thesis (in French), University of Toulouse, 2013.

\bibitem{claeys2014reconstruction}
M. Claeys, R. J. Sepulchre.
\newblock Reconstructing trajectories from the moments of occupation measures.
\newblock Proceedings of the IEEE Conference on Decision and Control, 2014.

\bibitem{impulse}
M. Claeys, D. Arzelier, D. Henrion, J.-B. Lasserre.
\newblock Measures and LMIs for non-linear optimal impulsive control.
\newblock  IEEE Transactions on Automatic Control, 59(5):1374-1379, 2014. 

\bibitem{diperna1987oscillations}
R. J. DiPerna, A. J. Majda.
\newblock Oscillations and concentrations in weak solutions of the
  incompressible fluid equations.
\newblock  Communications in Mathematical Physics, 108(4):667-689, 1987.

\bibitem{dykhta2010hamilton}
V. A. Dykhta, O. N. Samsonyuk.
\newblock Hamilton-Jacobi inequalities in control problems for impulsive
  dynamical systems.
\newblock Proceedings of the Steklov Institute of Mathematics,
  271(1):86--102, 2010.

\bibitem{Fattorini}
H. O. Fattorini.
\newblock Infinite dimensional optimization and control theory.
\newblock Cambridge University Press, UK, 1999.

\bibitem{GaitsgoryQuincampoix2009}
V. Gaitsgory, M. Quincampoix.
\newblock Linear programming approach to deterministic infinite horizon
optimal control problems with discounting.
\newblock SIAM Journal On Control and Optimization, 48(4):2480-2512, 2009. 

\bibitem{Gamkrelidze}
R. V. Gamkrelidze.
\newblock Principles of optimal control theory.
\newblock Plenum Press, New York, 1978.

\bibitem{getz-martin}
W. M. Getz, D.H. Martin.
\newblock Optimal control systems with state variable jump discontinuities.
\newblock Journal of Optimization Theory and Applications, 31:195--205, 1980.

\bibitem{roa}
D.~Henrion, M.~Korda.
\newblock Convex computation of the region of attraction of polynomial control
  systems.
\newblock IEEE Transactions on Automatic Control, 59(2):297--312, 2014.

\bibitem{GloptiPoly}
D.~Henrion, J.-B. Lasserre, J.~L\"ofberg.
\newblock Gloptipoly 3: Moments, optimization and semidefinite programming.
\newblock Optimization Methods and Software, 24(4-5):761--779, 2009.

\bibitem{Kruzik-Roubicek}
M.~Kru{\v{z}}{\'\i}k, T.~Roub{\'\i}{\v{c}}ek.
\newblock On the measures of DiPerna and Majda.
\newblock   Mathematica Bohemica,  122:383--399, 1997.

\bibitem{Kruzik1998optimization}
M.~Kru{\v{z}}{\'\i}k, T.~Roub{\'\i}{\v{c}}ek.
\newblock Optimization problems with concentration and oscillation
  effects: relaxation theory and numerical approximation.
\newblock Numerical Functional Analysis and Optimization, 20(5-6):511--530, 1999.

\bibitem{cdc05}
 J. B. Lasserre, C. Prieur, D. Henrion.
\newblock Nonlinear optimal control: numerical approximation via moments and LMI relaxations.
\newblock Proc. IEEE Conf. Decision and Control and Europ. Control Conf.,
Sevilla, Spain, 2005.

\bibitem{sicon}
J.-B. Lasserre, D. Henrion, C. Prieur,  E. Tr{\'e}lat.
\newblock Nonlinear optimal control via occupation measures and
  {LMI} relaxations.
\newblock SIAM Journal on Control and Optimization, 47(4):1643--1666,
  2008.

\bibitem{lasserre}
J.-B. Lasserre.
\newblock Positive polynomials and their applications.
\newblock Imperial College Press, London, UK, 2010.

\bibitem{Meziat-etal}
R. Meziat, D. Patino, P. Pedregal.
\newblock An alternative approach for non-linear optimal control problems based on the method of moments.
\newblock Computational Optimization and Applications, 38(1):147--171, 2007.

\bibitem{Meziat-etal-2}
R. Meziat, T. Roub\'{i}\v{c}ek, D. Patino.
\newblock  Coarse-convex-compactification approach to numerical solution of nonconvex variational problems.
 \newblock Numerical Functional Analysis and Optimization,  31(4):460--488, 2010.

\bibitem{miller2003impulsive}
B. Miller, E. Ya. Rubinovich.
\newblock Impulsive control in continuous and discrete-continuous systems.
\newblock Springer, Berlin, 2003.

\bibitem{motta1995space}
M. Motta, F. Rampazzo.
\newblock Space-time trajectories of nonlinear systems driven by ordinary and
  impulsive controls.
\newblock Differential and Integral Equations, 8(2):269--288, 1995.

\bibitem{Neustadt}
L.~W. Neustadt.
\newblock Optimization, a moment problem and nonlinear programming.
\newblock SIAM Journal of Control, 2(1):33--53, 1964.

\bibitem{Roubicek}
T.~Roub{\'\i}{\v{c}}ek. 
\newblock Relaxation in optimization theory and
variational calculus. W. de Gruyter, Berlin, 1997.

\bibitem{Roubicek-Kruzik-2000}
T.~Roub{\'\i}{\v{c}}ek, M.~Kru{\v{z}}{\'\i}k. 
\newblock Adaptive approximation algorithm for relaxed optimization problems 
\newblock In Proceedings of "Fast Solutions of Discrete Optimization Problems" held in
WIAS, Berlin, May 8--12, 2000, (Eds. V. Schulz, K.-H. Hoffmann and R.H.W. Hoppe),
Birkh\"{a}ser, Basel, 2001

\bibitem{Royden}
H. L. Royden,  P.~Fitzpatrick.
\newblock Real analysis. 4th edition.
\newblock Prentice Hall, NJ, 2010.

\bibitem{soucek}
J. Sou\v{c}ek.
\newblock  Spaces of functions on domain $\Omega$, whose
$k$-th derivatives are measures defined on $\bar{\Omega}$. 
\newblock
\v Casopis Pro P\v estov\'an\'{\i} Matematiky, 97:10--46, 1972. 

\bibitem{Stein}
E. M. Stein, R. Shakarchi. Princeton lectures on analysis III. Real analysis: measure
theory, integration, and Hilbert spaces. Princeton University Press,
Princeton, NJ, 2005.

\bibitem{SeDuMi}
J.~F. Sturm.
\newblock Using {S}e{D}u{M}i 1.02, a {M}atlab toolbox for optimization over
  symmetric cones.
\newblock Optimization Methods and Software, 11--12:625--653, 1999.

\bibitem{Vinter1993ConvexDuality}
R.~Vinter.
\newblock Convex duality and nonlinear optimal control.
\newblock SIAM Journal on Control and Optimization, 31(2):518--538, 1993.

\bibitem{Vinter1978Equivalence}
R.~Vinter, R.~Lewis.
\newblock The equivalence of strong and weak formulations for certain problems
  in optimal control.
\newblock SIAM Journal on Control and Optimization, 16(4):546--570, 1978.

\bibitem{Young}
L.~C.~Young.
\newblock Lectures on the calculus of variations and optimal control
  theory.
\newblock W. B. Saunders Co., Philadelphia, NJ, 1969.

\end{thebibliography}
\end{document}